\documentclass[11pt]{amsart}
\usepackage{amsmath,amssymb,amsxtra,amsthm}

\theoremstyle{plain}
\newtheorem{theorem}{Theorem}
\newtheorem{corollary}[theorem]{Corollary}
\newtheorem{lemma}[theorem]{Lemma}

\theoremstyle{definition}

\newcommand{\B}{\mathbb}
\newcommand{\C}{\mathcal}

\newcommand{\ga}{\alpha}

\newcommand{\gd}{\delta}
\newcommand{\eps}{\varepsilon}

\newcommand{\gq}{\vartheta}

\newcommand{\tbf}{\textbf}

\begin{document}

\title[Rational approximations related to Fermat numbers]{On the rational approximation of the sum of the reciprocals of the Fermat numbers}
\thanks{The research of M.~Coons is supported by a Fields--Ontario Fellowship and NSERC}
\date{\today}
\author{Michael Coons}
\address{University of Waterloo, Department of Pure Mathematics, Waterloo, Ontario, N2L 3G1, Canada}
\email{mcoons@math.uwaterloo.ca}


\begin{abstract}
Let $\C{G}(z):=\sum_{n=0}^\infty z^{2^n}(1-z^{2^n})^{-1}$ denote the generating function of the ruler function, and $\C{F}(z):=\sum_{n=0}^\infty z^{2^n}(1+z^{2^n})^{-1}$; note that the special value $\C{F}(1/2)$ is the sum of the reciprocals of the Fermat numbers $F_n:=2^{2^n}+1$. The functions $\C{F}(z)$ and $\C{G}(z)$ as well as their special values have been studied by Mahler, Golomb, Schwarz, and Duverney; it is known that the numbers $\C{F}(\ga)$ and $\C{G}(\ga)$ are transcendental for all algebraic numbers $\ga$ which satisfy $0<\ga<1$.

For a sequence $\mathbf{u}$, denote the Hankel matrix $H_n^p(\mathbf{u}):=(u({p+i+j-2}))_{1\leqslant i,j\leqslant n}$. Let $\ga$ be a real number. The {\em irrationality exponent} $\mu(\ga)$ is defined as the supremum of the set of real numbers $\mu$ such that the inequality $|\ga-p/q|<q^{-\mu}$ has infinitely many solutions $(p,q)\in\B{Z}\times\B{N}.$

In this paper, we first prove that the determinants of $H_n^1(\mathbf{g})$ and $H_n^1(\mathbf{f})$ are nonzero for every $n\geqslant 1$. We then use this result to  prove that for $b\geqslant 2$ the irrationality exponents $\mu(\C{F}(1/b))$ and $\mu(\C{G}(1/b))$ are equal to $2$; in particular, the irrationality exponent of the sum of the reciprocals of the Fermat numbers is $2$. 
\end{abstract}

\maketitle


\section{Introduction}

For $n\geqslant 0$ the $n$th {\em Fermat number} is given by $F_n:=2^{2^n}+1.$ In 1963, Golomb \cite{G1963} proved that the sum of the reciprocals of the Fermat numbers is irrational and then in 2001 Duverney \cite{D2001} proved transcendence though these results were probably known to Mahler as early as the late 1920s \cite{M1929,M1930a,M1930b}. In the same paper, Golomb proved something substantially more general; he defined the functions \begin{equation}\label{FGfg}\C{F}(z):=\sum_{n\geqslant 1}f(n)z^n=\sum_{n=0}^\infty\frac{z^{2^n}}{1+z^{2^n}}\quad\mbox{and}\quad  \C{G}(z):=\sum_{n\geqslant 1}g(n)z^n=\sum_{n=0}^\infty\frac{z^{2^n}}{1-z^{2^n}},\end{equation} and showed that both $ \C{F}(1/b)$ and $ \C{G}(1/b)$ are irrational for all positive integers $b\geqslant 2$; note that the special value $\C{F}(1/2)$ corresponds to the sum of the reciprocals of the Fermat numbers. Indeed, it is known that for $b\geqslant 2$ all of $ \C{F}(1/b)$ and $ \C{G}(1/b)$ are transcendental (this follows from results of Mahler \cite{M1929,M1930a,M1930b}; see Schwarz \cite{S1967} and Coons \cite{C2011} for details).

Let $\ga$ be a real number. The {\em irrationality exponent} $\mu(\ga)$ is defined as the supremum of the set of real numbers $\mu$ such that the inequality $$\left|\ga-\frac{p}{q}\right|<\frac{1}{q^\mu}$$ has infinitely many solutions $(p,q)\in\B{Z}\times\B{N}.$ For example, Liouville \cite{L1844} proved that for any sequence $\mathbf{a}:=\{a(n)\}_{n\geqslant 0}$ with $a(n)\in\{0,1\}$ for all $n$ and $a(n)$ not eventually zero, the numbers $\ga(\mathbf{a}):=\sum_{n\geqslant 0}a(n)10^{-n!}$ have $\mu(\ga(\mathbf{a}))=\infty$, and Roth \cite{R1955} showed that if $\ga$ is an irrational algebraic number, then $\mu(\ga)=2$. Note also that $\mu(\ga)\geqslant 2$ for all irrational $\ga$. 

In this paper, considering special values of the above series, we prove the following result.

\begin{theorem}\label{main} Let $b\geqslant 2$ be a positive integer. Then $\mu(\C{G}(1/b))=\mu(\C{F}(1/b))=2.$ In particular, $$\mu\left(\sum_{n\geqslant 0}\frac{1}{2^{2^n}+1}\right)=2.$$
\end{theorem}

Our method of proof is based on a method used recently by Bugeaud to prove that the irrationality exponent of the Thue--Morse--Mahler number is 2. To formalize, the Thue--Morse--Mahler sequence $\mathbf{t}:=\{t(n)\}_{n\geqslant 0}$ is defined by $t(0)=0$ and for $k\geqslant 0$, $t(2k)=t(k)$ and $t(2k+1)=1-t(k)$, and denote by $\C{T}(z)$ the generating function $$\C{T}(z)=\sum_{k\geqslant 0}t(k)z^k.$$ Bugeaud \cite{Bpreprint} proved for every $b\geqslant 2$ that $\mu(\C{T}(1/b))=2$. To do this, Bugeaud exploited a link between Pad\'e approximants and Hankel matrices (this connection is recorded as Lemma \ref{HP} of Section 3 of this paper) combined with a result of Allouche, Peyri\'ere, Wen and Wen \cite[Theorem 2.1]{APWW1998}, to provide a good rational approximation to the generating function $\C{T}(z)$, which was in turn used to prove his result.

For a sequence $\mathbf{u}=\{u(j)\}_{j\geqslant 0}$, we define the {\em Hankel matrix} $$H_n^p(\mathbf{u}):= (u({p+i+j-2}))_{1\leqslant i,j\leqslant n}.$$

The outline of this paper is as follows. In Section \ref{Sec2} we prove  

\begin{theorem}\label{mainH} Let $\mathbf{g}:=\{g(n)\}_{n\geqslant 1}$ and $\mathbf{f}:=\{f(n)\}_{n\geqslant 1}$ be the sequences defined in \eqref{FGfg}. The determinants of the Hankel matrices $H_n^1(\mathbf{g})$ and $H_n^1(\mathbf{f})$ are all nonzero.
\end{theorem}

\noindent In Section \ref{Sec3} we use this result, via a link with Pad\'e approximants, to prove Theorem~\ref{main}.

\section{Hankel determinants and the ruler function}\label{Sec2}

Note that if $\mathbf{g}=\{g(n)\}_{n\geqslant 1}$ is the sequence given in \eqref{FGfg}, then $g(n)$ is equal to the $2$--adic valuation of $2n$, known sometimes as the {\em ruler function}. For $n\geqslant 1$, the function $g(n)$ satisfies the recurrences $$g(2n+1)=1\qquad\mbox{and}\qquad g(2n)=1+g(n).$$ This sequence starts $$\mathbf{g}:=\{g(n)\}_{n\geqslant 1}=\{1,2,1,3,1,2,1,4,1,2,1,3,1,2,1,5,1,2,1,3,1,2,1,4,\ldots\}.$$ Since we will be working modulo $2$, we have two choices for $g(0)$, and we will have to make both; thus, let $\mathbf{g}^0:=\{0,g(1),g(2),\cdots\}$ be the sequence starting at $0$, and $\mathbf{g}^1:=\{1,g(1),g(2),\cdots\}$ be the sequence starting at $1$. 

We will need the following definitions and lemmas of Allouche, Peyri\`ere, Wen and Wen \cite{APWW1998}. The matrix $\mathbf{1}_{m\times n}$ is the $m\times n$ matrix with all its entries equal to $1$, and $\mathbf{0}_{m\times n}$ is the $m\times n$ matrix with all its entries equal to $0$. For the $n\times n$ square matrix $A$, we write $|A|$ and $A^t$ for the determinant of $A$ and the transpose of $A$, respectively, $\overline{A}$ for the matrix defined by $$\overline{A}:=\left(\begin{matrix} A & \mathbf{1}_{n\times 1}\\ \mathbf{1}_{n\times 1} & 0\end{matrix}\right),$$ and $A^{(j)}$ for the $n\times (n-1)$ matrix obtained by deleting the $j$th column of $A$. We write $I_n$ for the $n\times n$ identity matrix, and $$P_1(n)=(e_1,e_3,\ldots,e_{2\left\lfloor\frac{n+1}{2}\right\rfloor-1},e_2,e_4,\ldots,e_{2\left\lfloor\frac{n}{2}\right\rfloor}),$$ where $e_j$ is the column vector of length $n$ with a $1$ in its $j$th entry and zeros in the other entries. For convenience, throughout this paper we will write ``$\equiv$'' for equivalence modulo $2$.

The main result of this section is the following theorem.

\begin{theorem}\label{Hankelg} For all $n\geqslant 1$ we have \begin{align*} |H_n^{0}(\mathbf{g}^0)|\equiv|H_n^{2}(\mathbf{g}^0)|\equiv|H_n^{2}(\mathbf{g}^1)|&\equiv\begin{cases} 0 & \mbox{if $n\equiv 1,4\ ({\rm mod}\ 6)$}\\ 1 & \mbox{if $n\equiv 0,2,3,5\ ({\rm mod}\ 6)$},\end{cases}\\
|\overline{H_n^{0}(\mathbf{g}^0)}|&\equiv \begin{cases} 0 & \mbox{if $n\equiv 2,3\ ({\rm mod}\ 6)$}\\ 1 & \mbox{if $n\equiv 0,1,4,5\ ({\rm mod}\ 6)$},\end{cases}\\
|H_n^{0}(\mathbf{g}^1)|&\equiv\begin{cases} 0 & \mbox{if $n\equiv 1,2,4,5\ ({\rm mod}\ 6)$}\\ 1 & \mbox{if $n\equiv 0,3\ ({\rm mod}\ 6)$},\end{cases}\\
|\overline{H_n^{0}(\mathbf{g}^1)}|&\equiv \begin{cases} 0 & \mbox{if $n\equiv 0,1,2,3\ ({\rm mod}\ 6)$}\\ 1 & \mbox{if $n\equiv 4,5\ ({\rm mod}\ 6)$},\end{cases}\\
|H_n^{1}(\mathbf{g}^0)|\equiv |H_n^{1}(\mathbf{g}^1)|&\equiv 1,\\ 
|\overline{H_n^{1}(\mathbf{g}^0)}|\equiv|\overline{H_n^{1}(\mathbf{g}^1)}|&\equiv \begin{cases} 0 & \mbox{if $n\equiv 0,2,4\ ({\rm mod}\ 6)$}\\ 1 & \mbox{if $n\equiv 1,3,5\ ({\rm mod}\ 6)$},\end{cases}\\
|H_n^{2}(\mathbf{g}^0)|\equiv|H_n^{2}(\mathbf{g}^1)|&\equiv \begin{cases} 0 & \mbox{if $n\equiv 1,4\ ({\rm mod}\ 6)$}\\ 1 & \mbox{if $n\equiv 0,2,3,5\ ({\rm mod}\ 6)$},\end{cases}\\ 
|\overline{H_n^{2}(\mathbf{g}^0)}|\equiv |\overline{H_n^{2}(\mathbf{g}^1)}|&\equiv \begin{cases} 0 & \mbox{if $n\equiv 0,5\ ({\rm mod}\ 6)$}\\ 1 & \mbox{if $n\equiv 1,2,3,4\ ({\rm mod}\ 6)$}.\end{cases}
\end{align*}
\end{theorem}

To prove Theorem \ref{Hankelg} we will rely heavily in the following three lemmas, which originally occurred as Lemmas 1.2, 1.3, and 1.4 of \cite{APWW1998}.

\begin{lemma}[(Allouche et al.~\cite{APWW1998})]\label{2by2} Let $A$ and $B$ be two square matrices of order $m$ and $n$ respectively, and $a,b,x$ and $y$ for numbers. One has $$\left|\begin{matrix} aA & y\mathbf{1}_{m\times n}\\ x\mathbf{1}_{n\times m} & bB\end{matrix}\right|=a^mb^n|A|\cdot|B|-xya^{m-1}b^{m-1}|\overline{A}|\cdot|\overline{B}|.$$
\end{lemma}

\begin{lemma}[(Allouche et al.~\cite{APWW1998})]\label{3by3} Let $A$, $B$, and $C$ be three square matrices of order $m$, $n$, and $p$ respectively, and three numbers $a$, $b$, and $c$. One has \begin{align*}\left|\begin{matrix}A & c\mathbf{1}_{m\times n} & b\mathbf{1}_{m\times p}\\ c\mathbf{1}_{n\times m} & B & a\mathbf{1}_{n\times p}\\ b\mathbf{1}_{p\times m} & a\mathbf{1}_{p\times n} & C\end{matrix}\right|=|A|\cdot|B|\cdot&|C| -a^2|A|\cdot|\overline{B}|\cdot|\overline{C}|-b^2|\overline{A}|\cdot|B|\cdot|\overline{C}|\\ 
&-c^2|\overline{A}|\cdot|\overline{B}|\cdot|C|-2abc|\overline{A}|\cdot|\overline{B}|\cdot|\overline{C}|.\end{align*}
\end{lemma} 

\begin{lemma}[(Allouche et al.~\cite{APWW1998})]\label{AbarA} Let $x\in\B{R}$ and $A$ be an $m\times m$ matrix, then \begin{quote}\begin{enumerate}
\item[(i)] $|x\mathbf{1}_{m\times m}+A|=|A|-x|\overline{A}|,$
\vspace{.2cm}
\item[(ii)] $|\overline{x\mathbf{1}_{m\times m}+A}|=|\overline{A}|,$
\vspace{.2cm}
\item[(iii)] $|\overline{-A}|=(-1)^{m+1}|\overline{A}|.$
\end{enumerate}
\end{quote}
\end{lemma}

For a sequence $\mathbf{u}=\{u(j)\}_{j\geqslant 0}$ define the matrix $K_n^p(\mathbf{u})$ by $$K_n^p(\mathbf{u}):=(u({p+2(i+j-2)}))_{1\leqslant i,j\leqslant n}.$$

\begin{lemma}\label{mainlemma} For all $n\geqslant 1$, we have 
\begin{enumerate}
\item[(i')] $|H_{2n}^{0}(\mathbf{g}^1)|=|H_{n}^{0}(\mathbf{g}^0)|\cdot|H_{n}^{1}(\mathbf{g}^1)|-|\overline{H_{n}^{0}(\mathbf{g}^0)}|\cdot|H_{n}^{1}(\mathbf{g}^1)|-|H_{n}^{0}(\mathbf{g}^0)|\cdot|\overline{H_{n}^{1}(\mathbf{g}^1)}|,$
\vspace{.1cm}
\item[] $|H_{2n}^{0}(\mathbf{g}^0)|=|H_{n}^{0}(\mathbf{g}^1)|\cdot|H_{n}^{1}(\mathbf{g}^1)|-|\overline{H_{n}^{0}(\mathbf{g}^1)}|\cdot|H_{n}^{1}(\mathbf{g}^1)|-|H_{n}^{0}(\mathbf{g}^1)|\cdot|\overline{H_{n}^{1}(\mathbf{g}^1)}|,$
\vspace{.1cm}
\item[(i'')] for $p\geqslant 1$, \begin{align*} |H_{2n}^{2p}(\mathbf{g}^1)|=|H_{n}^{p}(\mathbf{g}^1)|\cdot&|H_{n}^{p+1}(\mathbf{g}^1)|-|\overline{H_{n}^{p}(\mathbf{g}^1)}|\cdot|H_{n}^{p+1}(\mathbf{g}^1)|\\ &-|H_{n}^{p}(\mathbf{g}^1)|\cdot|\overline{H_{n}^{p+1}(\mathbf{g}^1)}|,\end{align*}
\item[(ii')] $|\overline{H_{2n}^{0}(\mathbf{g}^1)}|\equiv |H_{n}^{0}(\mathbf{g}^0)|\cdot |\overline{H_{n}^{1}(\mathbf{g}^1)}|+|\overline{H_{n}^{0}(\mathbf{g}^0)}|\cdot|H_{n}^{1}(\mathbf{g}^1)|,$
\vspace{.1cm}
\item[] $|\overline{H_{2n}^{0}(\mathbf{g}^0)}|\equiv |H_{n}^{0}(\mathbf{g}^1)|\cdot |\overline{H_{n}^{1}(\mathbf{g}^1)}|+|\overline{H_{n}^{0}(\mathbf{g}^1)}|\cdot|H_{n}^{1}(\mathbf{g}^1)|,$
\item[(ii'')] for $p\geqslant 1$, $|\overline{H_{2n}^{2p}(\mathbf{g}^1)}|\equiv |H_{n}^{p}(\mathbf{g}^1)|\cdot |\overline{H_{n}^{p+1}(\mathbf{g}^1)}|+|\overline{H_{n}^{p}(\mathbf{g}^1)}|\cdot|H_{n}^{p+1}(\mathbf{g}^1)|$
\item[(iii')] $|H_{2n+1}^{0}(\mathbf{g}^1)|=|H_{n+1}^{0}(\mathbf{g}^0)|\cdot |H_{n}^{1}(\mathbf{g}^1)|-|\overline{H_{n+1}^{0}(\mathbf{g}^0)}|\cdot|H_{n}^{1}(\mathbf{g}^1)|$
\item[] $\qquad\qquad\qquad\qquad-|H_{n+1}^{0}(\mathbf{g}^0)|\cdot|\overline{H_{n}^{1}(\mathbf{g}^1)}|$
\item[] $|H_{2n+1}^{0}(\mathbf{g}^0)|=|H_{n+1}^{0}(\mathbf{g}^1)|\cdot |H_{n}^{1}(\mathbf{g}^1)|-|\overline{H_{n+1}^{0}(\mathbf{g}^1)}|\cdot|H_{n}^{1}(\mathbf{g}^1)|$
\item[] $\qquad\qquad\qquad\qquad-|H_{n+1}^{0}(\mathbf{g}^1)|\cdot|\overline{H_{n}^{1}(\mathbf{g}^1)}|$,
\item[(iii'')] for $p\geqslant 1,$ \begin{align*}|H_{2n+1}^{2p}(\mathbf{g}^1)|=|H_{n+1}^{p}(\mathbf{g}^1)|\cdot &|H_{n}^{p+1}(\mathbf{g}^1)|-|\overline{H_{n+1}^{p}(\mathbf{g}^1)}|\cdot|H_{n}^{p+1}(\mathbf{g}^1)|\\ &-|H_{n+1}^{p}(\mathbf{g}^1)|\cdot|\overline{H_{n}^{p+1}(\mathbf{g}^1)}|,\end{align*}
\item[(iv')] $|\overline{H_{2n+1}^{0}(\mathbf{g}^1)}|\equiv |H_{n+1}^{0}(\mathbf{g}^0)|\cdot |\overline{H_{n}^{1}(\mathbf{g}^1)}|+|\overline{H_{n+1}^{0}(\mathbf{g}^0)}|\cdot|H_{n}^{1}(\mathbf{g}^1)|$,
\vspace{.1cm}
\item[] $|\overline{H_{2n+1}^{0}(\mathbf{g}^0)}|\equiv |H_{n+1}^{0}(\mathbf{g}^1)|\cdot |\overline{H_{n}^{1}(\mathbf{g}^1)}|+|\overline{H_{n+1}^{0}(\mathbf{g}^1)}|\cdot|H_{n}^{1}(\mathbf{g}^1)|$,
\vspace{.1cm}
\item[(iv'')] for $p\geqslant 1$, $$|\overline{H_{2n+1}^{2p}(\mathbf{g}^1)}|\equiv |H_{n+1}^{p}(\mathbf{g}^1)|\cdot |\overline{H_{n}^{p+1}(\mathbf{g}^1)}|+|\overline{H_{n+1}^{p}(\mathbf{g}^1)}|\cdot|H_{n}^{p+1}(\mathbf{g}^1)|,$$
\item[(v)] for $p\geqslant 0$, $|H_{2n}^{2p+1}(\mathbf{g}^1)|\equiv |H_{n}^{p+1}(\mathbf{g}^1)|$,
\vspace{.1cm}
\item[(vi)] for $p\geqslant 0$, $|\overline{H_{2n}^{2p+1}(\mathbf{g}^1)}|\equiv 0$,
\vspace{.1cm}
\item[(vii')] $|H_{2n+1}^{1}(\mathbf{g}^1)|\equiv \Big[\Big(|H_{n+1}^{0}(\mathbf{g}^0)|\cdot|H_{n}^{1}(\mathbf{g}^1)|-|\overline{H_{n+1}^{0}(\mathbf{g}^0)}|\cdot|H_{n}^{1}(\mathbf{g}^1)|$
\item[] $\qquad\qquad\qquad\qquad-|H_{n+1}^{0}(\mathbf{g}^0)|\cdot|\overline{H_{n}^{1}(\mathbf{g}^1)}|\Big)$
\vspace{-.4cm}
\item[] \begin{align*}\times\Big(|&H_{n+1}^{1}(\mathbf{g}^1)|\cdot|H_{n}^{2}(\mathbf{g}^1)|-|\overline{H_{n+1}^{1}(\mathbf{g}^1)}|\cdot|H_{n}^{2}(\mathbf{g}^1)|-|H_{n+1}^{1}(\mathbf{g}^1)|\cdot|\overline{H_{n}^{2}(\mathbf{g}^1)}|\Big)\Big]\\
&\qquad-\Big[\Big(|H_{n}^{1}(\mathbf{g}^1)|\cdot|H_{n}^{2}(\mathbf{g}^1)|-|\overline{H_{n}^{1}(\mathbf{g}^1)}|\cdot|H_{n}^{2}(\mathbf{g}^1)|-|H_{n}^{1}(\mathbf{g}^1)|\cdot|\overline{H_{n}^{2}(\mathbf{g}^1)}|\Big)\\
\times\Big(|&H_{n+1}^{0}(\mathbf{g}^0)|\cdot|H_{n+1}^{1}(\mathbf{g}^1)|-|\overline{H_{n+1}^{0}(\mathbf{g}^0)}|\cdot|H_{n+1}^{1}(\mathbf{g}^1)|\\
&\qquad\qquad\qquad-|H_{n+1}^{0}(\mathbf{g}^0)|\cdot|\overline{H_{n}^{1}(\mathbf{g}^1)}|\Big)\Big],
\end{align*}
\item[] $|H_{2n+1}^{1}(\mathbf{g}^0)|\equiv \Big[\Big(|H_{n+1}^{0}(\mathbf{g}^1)|\cdot|H_{n}^{1}(\mathbf{g}^1)|-|\overline{H_{n+1}^{0}(\mathbf{g}^1)}|\cdot|H_{n}^{1}(\mathbf{g}^1)|$
\item[] $\qquad\qquad\qquad\qquad-|H_{n+1}^{0}(\mathbf{g}^1)|\cdot|\overline{H_{n+1}^{1}(\mathbf{g}^1)}|\Big)$
\vspace{-.4cm}
\item[] \begin{align*}\times\Big(|&H_{n+1}^{1}(\mathbf{g}^1)|\cdot|H_{n}^{2}(\mathbf{g}^1)|-|\overline{H_{n+1}^{1}(\mathbf{g}^1)}|\cdot|H_{n}^{2}(\mathbf{g}^1)|-|H_{n+1}^{1}(\mathbf{g}^1)|\cdot|\overline{H_{n}^{2}(\mathbf{g}^1)}|\Big)\Big]\\
&\qquad-\Big[\Big(|H_{n}^{1}(\mathbf{g}^1)|\cdot|H_{n}^{2}(\mathbf{g}^1)|-|\overline{H_{n}^{1}(\mathbf{g}^1)}|\cdot|H_{n}^{2}(\mathbf{g}^1)|-|H_{n}^{1}(\mathbf{g}^1)|\cdot|\overline{H_{n}^{2}(\mathbf{g}^1)}|\Big)\\
\times\Big(|&H_{n+1}^{0}(\mathbf{g}^1)|\cdot|H_{n+1}^{1}(\mathbf{g}^1)|-|\overline{H_{n+1}^{0}(\mathbf{g}^1)}|\cdot|H_{n+1}^{1}(\mathbf{g}^1)|\\
&\qquad\qquad\qquad-|H_{n+1}^{0}(\mathbf{g}^1)|\cdot|\overline{H_{n+1}^{1}(\mathbf{g}^1)}|\Big)\Big],
\end{align*}
\item[(vii'')] for $p\geqslant 1$, \begin{align*}|H_{2n+1}^{2p+1}(\mathbf{g}^1)|\equiv &\left(|H_{n}^{p+2}(\mathbf{g}^1)|\cdot|\overline{H_{n+1}^{p}(\mathbf{g}^1)}|-|H_{n+1}^{p}(\mathbf{g}^1)|\cdot|\overline{H_{n}^{p+2}(\mathbf{g}^1)}|\right)\\
&\qquad\times\left(|H_{n}^{p+1}(\mathbf{g}^1)|\cdot|\overline{H_{n+1}^{p+1}(\mathbf{g}^1)}|-|H_{n+1}^{p+1}(\mathbf{g}^1)|\cdot|\overline{H_{n}^{p+1}(\mathbf{g}^1)}|\right),\end{align*}
\item[(viii')] $ |\overline{H_{2n+1}^{1}(\mathbf{g}^1)}| \equiv \left(|H_{n}^{2}(\mathbf{g}^1)|\cdot|\overline{H_{n+1}^{0}(\mathbf{g}^1)}|-|H_{n+1}^{0}(\mathbf{g}^1)|\cdot|\overline{H_{n}^{2}(\mathbf{g}^1)}|\right)$
\vspace{-.5cm} 
\item[] \begin{align*} &\qquad\qquad\qquad\qquad\times\left(|H_{n}^{1}(\mathbf{g}^1)|\cdot|\overline{H_{n+1}^{1}(\mathbf{g}^1)}|-|H_{n+1}^{1}(\mathbf{g}^1)|\cdot|\overline{H_{n}^{1}(\mathbf{g}^1)}|\right)\\
&\qquad\qquad+|H_{n+1}^{0}(\mathbf{g}^0)|\cdot|H_{n}^{1}(\mathbf{g}^1)|\cdot|H_{n+1}^{1}(\mathbf{g}^1)|\cdot |H_{n}^{2}(\mathbf{g}^1)|\\ 
&\qquad\qquad+|H_{n}^{1}(\mathbf{g}^1)|\cdot |H_{n}^{2}(\mathbf{g}^1)|\cdot|H_{n+2}^{0}(\mathbf{g}^0)|\cdot|\overline{H_{n+2}^{1}(\mathbf{g}^1)}|,\end{align*}
\item[(viii'')] for $p\geqslant 1$ \begin{align*}|\overline{H_{2n+1}^{2p+1}(\mathbf{g}^1)}| &\equiv \left(|H_{n}^{p+2}(\mathbf{g}^1)|\cdot|\overline{H_{n+1}^{p}(\mathbf{g}^1)}|-|H_{n+1}^{p}(\mathbf{g}^1)|\cdot|\overline{H_{n}^{p+2}(\mathbf{g}^1)}|\right)\\ &\qquad\times\left(|H_{n}^{p+1}(\mathbf{g}^1)|\cdot|\overline{H_{n+1}^{p+1}(\mathbf{g}^1)}|-|H_{n+1}^{p+1}(\mathbf{g}^1)|\cdot|\overline{H_{n}^{p+1}(\mathbf{g}^1)}|\right)\\
&\qquad+|H_{n+1}^{p}(\mathbf{g}^1)|\cdot|H_{n}^{p+1}(\mathbf{g}^1)|\cdot|H_{n+1}^{p+1}(\mathbf{g}^1)|\cdot |H_{n}^{p+2}(\mathbf{g}^1)|\\ 
&\qquad+|H_{n}^{p+1}(\mathbf{g}^1)|\cdot |H_{n}^{p+2}(\mathbf{g}^1)|\cdot|H_{n+2}^{p}(\mathbf{g}^1)|\cdot|H_{n+2}^{p+1}(\mathbf{g}^1)|.\end{align*}
\end{enumerate}
\end{lemma}

\begin{proof} For $p\geqslant 1$ we have that \begin{multline}\label{KH1} K_n^{2p}(\mathbf{g}^0)=K_n^{2p}(\mathbf{g}^1)=(g({2p+2(i+j-2)}))_{1\leqslant i,j\leqslant n}\\ =(1+ g({p+(i+j-2)}))_{1\leqslant i,j\leqslant n} =\mathbf{1}_{n\times n}+ H_n^p(\mathbf{g}^1),\end{multline} and for $p\geqslant 0$ that \begin{multline}\label{KH2} K_n^{2p+1}(\mathbf{g}^1)=(g({2p+1+2(i+j-2)}))_{1\leqslant i,j\leqslant n}\\ =(g({2(p+i+j-2)+1}))_{1\leqslant i,j\leqslant n}=\mathbf{1}_{n\times n}.\end{multline} The analogue of \eqref{KH1} for $p=0$ must take into account the difference of $\mathbf{g}^1$ and $\mathbf{g}^0$ in there first coordinate. Since $g^0(0)\equiv 1+g^1(0)\ ({\rm mod}\ 2)$ and $g^1(0)\equiv 1+g^0(0)\ ({\rm mod}\ 2)$, we have that \begin{equation}\label{KH0}K_n^0(\mathbf{g}^0)=\mathbf{1}_{n\times n}+ H_n^0(\mathbf{g}^1),\qquad\mbox{and}\qquad K_n^0(\mathbf{g}^1)=\mathbf{1}_{n\times n}+ H_n^0(\mathbf{g}^0).\end{equation}

Note that for $P_1$ defined above we have (see equation (8) of \cite{APWW1998}) that \begin{equation}\label{APWW8} P_1^tH_{2n}^p(\mathbf{u})P_1=\left(\begin{matrix}K_{n}^{p}(\mathbf{u}) &K_{n}^{p+1}(\mathbf{u})\\ K_{n}^{p+1}(\mathbf{u}) & K_{n}^{p+2}(\mathbf{u}) \end{matrix}\right),\end{equation} and \begin{align} \label{APWW9} P_1^tH_{2n+1}^p(\mathbf{u})P_1&=\left(\begin{matrix}K_{n+1}^{p}(\mathbf{u}) &(K_{n+1}^{p+1}(\mathbf{u}))^{(n+1)}\\ (K_{n+1}^{p+1}(\mathbf{u}))^{(n+1)t} & K_{n}^{p+2}(\mathbf{u}) \end{matrix}\right).
\end{align} 

To prove (i) we must now break into two cases. If $p=0$, then using \eqref{APWW8}, \eqref{KH1} and \eqref{KH2}, we have that \begin{align*} P_1^t H_{2n}^{0}(\mathbf{g}^1)P_1&=\left(\begin{matrix}K_{n}^{0}(\mathbf{g}^1) &K_{n}^{1}(\mathbf{g}^1)\\ K_{n}^{1}(\mathbf{g}^1) & K_{n}^{2}(\mathbf{g}^1) \end{matrix}\right)\\
&=\left(\begin{matrix}\mathbf{1}_{n\times n}+ H_{n}^{0}(\mathbf{g}^0) & \mathbf{1}_{n\times n}\\ \mathbf{1}_{n\times n} & \mathbf{1}_{n\times n}+ H_{n}^{1}(\mathbf{g}^1) \end{matrix}\right),
\end{align*} so that Lemma \ref{2by2} gives \begin{align*} |H_{2n}^{0}(\mathbf{g}^1)|&=|\mathbf{1}_{n\times n}+ H_{n}^{0}(\mathbf{g}^0)|\cdot|\mathbf{1}_{n\times n}+ H_{n}^{1}(\mathbf{g}^1)|\\
&\qquad\qquad\qquad-|\overline{\mathbf{1}_{n\times n}+ H_{n}^{0}(\mathbf{g}^0)}|\cdot|\overline{\mathbf{1}_{n\times n}+ H_{n}^{1}(\mathbf{g}^1)}|\\
&=|H_{n}^{0}(\mathbf{g}^0)|\cdot|H_{n}^{1}(\mathbf{g}^1)|-|\overline{H_{n}^{0}(\mathbf{g}^0)}|\cdot|H_{n}^{1}(\mathbf{g}^1)|-|H_{n}^{0}(\mathbf{g}^0)|\cdot|\overline{H_{n}^{1}(\mathbf{g}^1)}| 
\end{align*} The similar result holds for $|H_{2n}^{0}(\mathbf{g}^0)|$ by replacing $\mathbf{g}^1$ with $\mathbf{g}^0$ in the above argument. This proves~(i').

If $p\geqslant 1$, then again that using \eqref{APWW8}, \eqref{KH1} and \eqref{KH2}, we have that \begin{align*} P_1^t H_{2n}^{2p}(\mathbf{g}^1)P_1&=\left(\begin{matrix}K_{n}^{2p}(\mathbf{g}^1) &K_{n}^{2p+1}(\mathbf{g}^1)\\ K_{n}^{2p+1}(\mathbf{g}^1) & K_{n}^{2p+2}(\mathbf{g}^1) \end{matrix}\right)\\
&=\left(\begin{matrix}\mathbf{1}_{n\times n}+ H_{n}^{p}(\mathbf{g}^1) & \mathbf{1}_{n\times n}\\ \mathbf{1}_{n\times n} & \mathbf{1}_{n\times n}+ H_{n}^{p+1}(\mathbf{g}^1) \end{matrix}\right),
\end{align*} so that Lemma \ref{2by2} gives \begin{align*} |H_{2n}^{2p}(\mathbf{g}^1)|&=|\mathbf{1}_{n\times n}+ H_{n}^{p}(\mathbf{g}^1)|\cdot|\mathbf{1}_{n\times n}+ H_{n}^{p+1}(\mathbf{g}^1)|\\
&\qquad\qquad\qquad-|\overline{\mathbf{1}_{n\times n}+ H_{n}^{p}(\mathbf{g}^1)}|\cdot|\overline{\mathbf{1}_{n\times n}+ H_{n}^{p+1}(\mathbf{g}^1)}|\\
&=|H_{n}^{p}(\mathbf{g}^1)|\cdot|H_{n}^{p+1}(\mathbf{g}^1)|-|\overline{H_{n}^{p}(\mathbf{g}^1)}|\cdot|H_{n}^{p+1}(\mathbf{g}^1)|-|H_{n}^{p}(\mathbf{g}^1)|\cdot|\overline{H_{n}^{p+1}(\mathbf{g}^1)}|, 
\end{align*} which proves (i'').

For (ii'), we have that \begin{align*} \left(\begin{matrix}P_1^t & \mathbf{0}_{2n\times 1}\\ \mathbf{0}_{1\times 2n} & 1 \end{matrix}\right) &\overline{H_{2n}^{0}(\mathbf{g}^1)} \left(\begin{matrix}P_1 & \mathbf{0}_{2n\times 1}\\ \mathbf{0}_{1\times 2n} & 1 \end{matrix}\right) \\
&=  \left(\begin{matrix}P_1^t & \mathbf{0}_{2n\times 1}\\ \mathbf{0}_{1\times 2n} & 1 \end{matrix}\right) \left(\begin{matrix} H_{2n}^{0}(\mathbf{g}^1) & \mathbf{1}_{2n\times 1}\\ \mathbf{1}_{1\times 2n} & 0\end{matrix}\right) \left(\begin{matrix}P_1 & \mathbf{0}_{2n\times 1}\\ \mathbf{0}_{1\times 2n} & 1 \end{matrix}\right)\\
&=\left(\begin{matrix}P_1^t H_{2n}^{0}(\mathbf{g}^1)P_1 & P_1^t\mathbf{1}_{2n\times 1}\\ \mathbf{1}_{1\times 2n}P_1 & 0\end{matrix}\right)\\
&=\left(\begin{matrix}\mathbf{1}_{n\times n}+ H_{n}^{0}(\mathbf{g}^0) & \mathbf{1}_{n\times n} & \mathbf{1}_{n\times 1}\\ \mathbf{1}_{n\times n} & \mathbf{1}_{n\times n}+ H_{n}^{1}(\mathbf{g}^1) & \mathbf{1}_{n\times 1}\\ \mathbf{1}_{1\times n} & \mathbf{1}_{1\times n} & 0\end{matrix}\right).
\end{align*} If we consider the $1\times 1$ matrix $(0)$, then $$\overline{(0)}=\left(\begin{matrix}0 & 1\\ 1& 0\end{matrix}\right)$$ so that $|(0)|=0$ and $|\overline{(0)}|=-1$. Thus Lemma \ref{3by3} gives \begin{align*}|\overline{H_{2n}^{0}(\mathbf{g}^1)}|&=|\mathbf{1}_{n\times n}+ H_{n}^{0}(\mathbf{g}^0)|\cdot |\overline{\mathbf{1}_{n\times n}+ H_{n}^{1}(\mathbf{g}^1)}|\\
&\qquad\qquad+|\overline{\mathbf{1}_{n\times n}+ H_{n}^{0}(\mathbf{g}^0)}|\cdot |\mathbf{1}_{n\times n}+ H_{n}^{1}(\mathbf{g}^1)|\\
&\qquad\qquad\qquad+2|\overline{\mathbf{1}_{n\times n}+ H_{n}^{0}(\mathbf{g}^0)}|\cdot |\overline{\mathbf{1}_{n\times n}+ H_{n}^{1}(\mathbf{g}^1)}|\\
&=\left(|H_{n}^{0}(\mathbf{g}^0)|-|\overline{H_{n}^{0}(\mathbf{g}^0)}|\right)\cdot |\overline{H_{n}^{1}(\mathbf{g}^1)}|+|\overline{H_{n}^{0}(\mathbf{g}^0)}|\cdot\left(|H_{n}^{1}(\mathbf{g}^1)|-|\overline{H_{n}^{1}(\mathbf{g}^1)}|\right)\\
&\qquad\qquad+2  |\overline{H_{n}^{0}(\mathbf{g}^0)}|\cdot |\overline{H_{n}^{1}(\mathbf{g}^1)}|\\
&\equiv |H_{n}^{0}(\mathbf{g}^0)|\cdot |\overline{H_{n}^{1}(\mathbf{g}^1)}|+|\overline{H_{n}^{0}(\mathbf{g}^0)}|\cdot|H_{n}^{1}(\mathbf{g}^1)|.
\end{align*} The similar result holds for $|\overline{H_{2n}^{0}(\mathbf{g}^0)}|$ by replacing $\mathbf{g}^1$ with $\mathbf{g}^0$ in the above argument. This proves~(ii').

For (ii''), note that for $p\geqslant 1$ we have \begin{align*} \left(\begin{matrix}P_1^t & \mathbf{0}_{2n\times 1}\\ \mathbf{0}_{1\times 2n} & 1 \end{matrix}\right) &\overline{H_{2n}^{2p}(\mathbf{g}^1)} \left(\begin{matrix}P_1 & \mathbf{0}_{2n\times 1}\\ \mathbf{0}_{1\times 2n} & 1 \end{matrix}\right) \\
&=  \left(\begin{matrix}P_1^t & \mathbf{0}_{2n\times 1}\\ \mathbf{0}_{1\times 2n} & 1 \end{matrix}\right) \left(\begin{matrix} H_{2n}^{2p}(\mathbf{g}^1) & \mathbf{1}_{2n\times 1}\\ \mathbf{1}_{1\times 2n} & 0\end{matrix}\right) \left(\begin{matrix}P_1 & \mathbf{0}_{2n\times 1}\\ \mathbf{0}_{1\times 2n} & 1 \end{matrix}\right)\\
&=\left(\begin{matrix}P_1^t H_{2n}^{2p}(\mathbf{g}^1)P_1 & P_1^t\mathbf{1}_{2n\times 1}\\ \mathbf{1}_{1\times 2n}P_1 & 0\end{matrix}\right)\\
&=\left(\begin{matrix}\mathbf{1}_{n\times n}+ H_{n}^{p}(\mathbf{g}^1) & \mathbf{1}_{n\times n} & \mathbf{1}_{n\times 1}\\ \mathbf{1}_{n\times n} & \mathbf{1}_{n\times n}+ H_{n}^{p+1}(\mathbf{g}^1) & \mathbf{1}_{n\times 1}\\ \mathbf{1}_{1\times n} & \mathbf{1}_{1\times n} & 0\end{matrix}\right).
\end{align*} Using the above comments about $|(0)|$ and $|\overline{(0)}|$ and  Lemma \ref{3by3} gives \begin{align*}|\overline{H_{2n}^{2p}(\mathbf{g}^1)}|&=|\mathbf{1}_{n\times n}+ H_{n}^{p}(\mathbf{g}^1)|\cdot |\overline{\mathbf{1}_{n\times n}+ H_{n}^{p+1}(\mathbf{g}^1)}|\\
&\qquad\qquad+|\overline{\mathbf{1}_{n\times n}+ H_{n}^{p}(\mathbf{g}^1)}|\cdot |\mathbf{1}_{n\times n}+ H_{n}^{p+1}(\mathbf{g}^1)|\\
&\qquad\qquad\qquad+2|\overline{\mathbf{1}_{n\times n}+ H_{n}^{p}(\mathbf{g}^1)}|\cdot |\overline{\mathbf{1}_{n\times n}+ H_{n}^{p+1}(\mathbf{g}^1)}|\\
&=\left(|H_{n}^{p}(\mathbf{g}^1)|-|\overline{H_{n}^{p}(\mathbf{g}^1)}|\right)\cdot |\overline{H_{n}^{p+1}(\mathbf{g}^1)}|\\
&\qquad\qquad+|\overline{H_{n}^{p}(\mathbf{g}^1)}|\cdot\left(|H_{n}^{p+1}(\mathbf{g}^1)|-|\overline{H_{n}^{p+1}(\mathbf{g}^1)}|\right)\\
&\qquad\quad\qquad+2  |\overline{H_{n}^{p}(\mathbf{g}^1)}|\cdot |\overline{H_{n}^{p+1}(\mathbf{g}^1)}|\\
&\equiv |H_{n}^{p}(\mathbf{g}^1)|\cdot |\overline{H_{n}^{p+1}(\mathbf{g}^1)}|+|\overline{H_{n}^{p}(\mathbf{g}^1)}|\cdot|H_{n}^{p+1}(\mathbf{g}^1)|.
\end{align*}

For (iii'), note that \begin{align}\nonumber P_1^tH_{2n+1}^{0}(\mathbf{g}^1)P_1&=\left(\begin{matrix}K_{n+1}^{0}(\mathbf{g}^1) &(K_{n+1}^{1}(\mathbf{g}^1))^{(n+1)}\\ (K_{n+1}^{1}(\mathbf{g}^1))^{(n+1)t} & K_{n}^{2}(\mathbf{g}^1) \end{matrix}\right)\\
&\nonumber =\left(\begin{matrix}\mathbf{1}_{(n+1)\times (n+1)}+ H_{n+1}^{0}(\mathbf{g}^0) &(\mathbf{1}_{(n+1)\times (n+1)})^{(n+1)}\\ (\mathbf{1}_{(n+1)\times (n+1)})^{(n+1)t} & \mathbf{1}_{n\times n}+ H_{n}^{1}(\mathbf{g}^1) \end{matrix}\right)\\
&\label{0H0221}=\left(\begin{matrix}\mathbf{1}_{(n+1)\times (n+1)}+ H_{n+1}^{0}(\mathbf{g}^0) &\mathbf{1}_{(n+1)\times n}\\ \mathbf{1}_{n\times (n+1)} & \mathbf{1}_{n\times n}+ H_{n}^{1}(\mathbf{g}^1) \end{matrix}\right). 
\end{align} Thus we have \begin{align*}|H_{2n+1}^{0}(\mathbf{g}^1)|&=|\mathbf{1}_{(n+1)\times (n+1)}+ H_{n+1}^{0}(\mathbf{g}^0)|\cdot |\mathbf{1}_{n\times n}+ H_{n}^{1}(\mathbf{g}^1)|\\
&\qquad\qquad-|\overline{\mathbf{1}_{(n+1)\times (n+1)}+ H_{n+1}^{0}(\mathbf{g}^0)}|\cdot |\overline{\mathbf{1}_{n\times n}+ H_{n}^{1}(\mathbf{g}^1)}|\\
&=\left(|H_{n+1}^{0}(\mathbf{g}^0)|-|\overline{H_{n+1}^{0}(\mathbf{g}^0)}|\right)\cdot\left(|H_{n}^{1}(\mathbf{g}^1)|-|\overline{H_{n}^{1}(\mathbf{g}^1)}|\right)\\
&\qquad\qquad-|\overline{H_{n+1}^{0}(\mathbf{g}^0)}|\cdot |\overline{H_{n}^{1}(\mathbf{g}^1)}|\\
&=|H_{n+1}^{0}(\mathbf{g}^0)|\cdot |H_{n}^{1}(\mathbf{g}^1)|-|\overline{H_{n+1}^{0}(\mathbf{g}^0)}|\cdot|H_{n}^{1}(\mathbf{g}^1)|\\
&\qquad\qquad-|H_{n+1}^{0}(\mathbf{g}^0)|\cdot|\overline{H_{n}^{1}(\mathbf{g}^1)}|.
\end{align*} The similar result holds for $|\overline{H_{2n}^{0}(\mathbf{g}^0)}|$ by replacing $\mathbf{g}^1$ with $\mathbf{g}^0$ in the above argument. This proves~(iii').

For (iii''), $p\geqslant 1$ so that \begin{align}\nonumber P_1^tH_{2n+1}^{2p}(\mathbf{g}^1)P_1&=\left(\begin{matrix}K_{n+1}^{2p}(\mathbf{g}^1) &(K_{n+1}^{2p+1}(\mathbf{g}^1))^{(n+1)}\\ (K_{n+1}^{2p+1}(\mathbf{g}^1))^{(n+1)t} & K_{n}^{2p+2}(\mathbf{g}^1) \end{matrix}\right)\\
&\nonumber =\left(\begin{matrix}\mathbf{1}_{(n+1)\times (n+1)}+ H_{n+1}^{p}(\mathbf{g}^1) &(\mathbf{1}_{(n+1)\times (n+1)})^{(n+1)}\\ (\mathbf{1}_{(n+1)\times (n+1)})^{(n+1)t} & \mathbf{1}_{n\times n}+ H_{n}^{p+1}(\mathbf{g}^1) \end{matrix}\right)\\
&\label{PHP221}=\left(\begin{matrix}\mathbf{1}_{(n+1)\times (n+1)}+ H_{n+1}^{p}(\mathbf{g}^1) &\mathbf{1}_{(n+1)\times n}\\ \mathbf{1}_{n\times (n+1)} & \mathbf{1}_{n\times n}+ H_{n}^{p+1}(\mathbf{g}^1) \end{matrix}\right). 
\end{align} Thus we have \begin{align*}|H_{2n+1}^{2p}(\mathbf{g}^1)|&=|\mathbf{1}_{(n+1)\times (n+1)}+ H_{n+1}^{p}(\mathbf{g}^1)|\cdot |\mathbf{1}_{n\times n}+ H_{n}^{p+1}(\mathbf{g}^1)|\\
&\qquad\qquad-|\overline{\mathbf{1}_{(n+1)\times (n+1)}+ H_{n+1}^{p}(\mathbf{g}^1)}|\cdot |\overline{\mathbf{1}_{n\times n}+ H_{n}^{p+1}(\mathbf{g}^1)}|\\
&=\left(|H_{n+1}^{p}(\mathbf{g}^1)|-|\overline{H_{n+1}^{p}(\mathbf{g}^1)}|\right)\cdot\left(|H_{n}^{p+1}(\mathbf{g}^1)|-|\overline{H_{n}^{p+1}(\mathbf{g}^1)}|\right)\\
&\qquad\qquad-|\overline{H_{n+1}^{p}(\mathbf{g}^1)}|\cdot |\overline{H_{n}^{p+1}(\mathbf{g}^1)}|\\
&=|H_{n+1}^{p}(\mathbf{g}^1)|\cdot |H_{n}^{p+1}(\mathbf{g}^1)|-|\overline{H_{n+1}^{p}(\mathbf{g}^1)}|\cdot|H_{n}^{p+1}(\mathbf{g}^1)|\\
&\qquad\qquad-|H_{n+1}^{p}(\mathbf{g}^1)|\cdot|\overline{H_{n}^{p+1}(\mathbf{g}^1)}|.
\end{align*}

For (iv'), similar to (ii') we have \begin{align*} &\left(\begin{matrix}P_1^t & \mathbf{0}_{(2n+1)\times 1}\\ \mathbf{0}_{1\times (2n+1)} & 1 \end{matrix}\right) \overline{H_{2n+1}^{0}(\mathbf{g}^1)} \left(\begin{matrix}P_1 & \mathbf{0}_{(2n+1)\times 1}\\ \mathbf{0}_{1\times (2n+1)} & 1 \end{matrix}\right) \\
&\qquad=  \left(\begin{matrix}P_1^t & \mathbf{0}_{(2n+1)\times 1}\\ \mathbf{0}_{1\times (2n+1)} & 1 \end{matrix}\right) \left(\begin{matrix} H_{2n+1}^{0}(\mathbf{g}^1) & \mathbf{1}_{(2n+1)\times 1}\\ \mathbf{1}_{1\times (2n+1)} & 0\end{matrix}\right) \left(\begin{matrix}P_1 & \mathbf{0}_{(2n+1)\times 1}\\ \mathbf{0}_{1\times (2n+1)} & 1 \end{matrix}\right)\\
&\qquad=\left(\begin{matrix}P_1^t H_{2n+1}^{0}(\mathbf{g}^1)P_1 & P_1^t\mathbf{1}_{(2n+1)\times 1}\\ \mathbf{1}_{1\times (2n+1)}P_1 & 0\end{matrix}\right).
\end{align*} Thus using \eqref{0H0221} we have 
\begin{align*} &\left(\begin{matrix}P_1^t & \mathbf{0}_{(2n+1)\times 1}\\ \mathbf{0}_{1\times (2n+1)} & 1 \end{matrix}\right) \overline{H_{2n+1}^{0}(\mathbf{g}^1)} \left(\begin{matrix}P_1 & \mathbf{0}_{(2n+1)\times 1}\\ \mathbf{0}_{1\times (2n+1)} & 1 \end{matrix}\right) \\
&\qquad\qquad\qquad=\left(\begin{matrix}\mathbf{1}_{(n+1)\times (n+1)}+ H_{n+1}^{0}(\mathbf{g}^0) & \mathbf{1}_{(n+1)\times n} & \mathbf{1}_{(n+1)\times 1}\\ \mathbf{1}_{n\times (n+1)} & \mathbf{1}_{n\times n}+ H_{n}^{1}(\mathbf{g}^1) & \mathbf{1}_{n\times 1}\\ \mathbf{1}_{1\times (n+1)} & \mathbf{1}_{1\times n} & 0\end{matrix}\right).
\end{align*} Just as before, we have $|(0)|=0$ and $|\overline{(0)}|=-1$, and so again using Lemma \ref{3by3} we have \begin{align*}|\overline{H_{2n+1}^{0}(\mathbf{g}^1)}|
&=|\mathbf{1}_{(n+1)\times (n+1)}+ H_{n+1}^{0}(\mathbf{g}^0)|\cdot |\overline{\mathbf{1}_{n\times n}+ H_{n}^{1}(\mathbf{g}^1)}|\\
&\qquad+|\overline{\mathbf{1}_{(n+1)\times (n+1)}+ H_{n+1}^{0}(\mathbf{g}^0)}|\cdot |\mathbf{1}_{n\times n}+ H_{n}^{1}(\mathbf{g}^1)|\\
&\qquad\qquad+2|\overline{\mathbf{1}_{(n+1)\times (n+1)}+ H_{n+1}^{0}(\mathbf{g}^0)}|\cdot |\overline{\mathbf{1}_{n\times n}+ H_{n}^{1}(\mathbf{g}^1)}|\\
&=\left(|H_{n+1}^{0}(\mathbf{g}^0)|-|\overline{H_{n+1}^{0}(\mathbf{g}^0)}|\right)\cdot |\overline{H_{n}^{1}(\mathbf{g}^1)}|\\
&\qquad+|\overline{H_{n+1}^{0}(\mathbf{g}^0)}|\cdot\left(|H_{n}^{1}(\mathbf{g}^1)|-|\overline{H_{n}^{1}(\mathbf{g}^1)}|\right)+2  |\overline{H_{n+1}^{0}(\mathbf{g}^0)}|\cdot |\overline{H_{n}^{1}(\mathbf{g}^1)}|\\
&\equiv |H_{n+1}^{0}(\mathbf{g}^0)|\cdot |\overline{H_{n}^{1}(\mathbf{g}^1)}|+|\overline{H_{n+1}^{0}(\mathbf{g}^0)}|\cdot|H_{n}^{1}(\mathbf{g}^1)|.
\end{align*} The similar result holds for $|\overline{H_{2n}^{0}(\mathbf{g}^0)}|$ by replacing $\mathbf{g}^1$ with $\mathbf{g}^0$ in the above argument. This proves~(iv').

For (iv''), similar to (ii'') we have \begin{align*} &\left(\begin{matrix}P_1^t & \mathbf{0}_{(2n+1)\times 1}\\ \mathbf{0}_{1\times (2n+1)} & 1 \end{matrix}\right) \overline{H_{2n+1}^{2p}(\mathbf{g}^1)} \left(\begin{matrix}P_1 & \mathbf{0}_{(2n+1)\times 1}\\ \mathbf{0}_{1\times (2n+1)} & 1 \end{matrix}\right) \\
&\qquad=  \left(\begin{matrix}P_1^t & \mathbf{0}_{(2n+1)\times 1}\\ \mathbf{0}_{1\times (2n+1)} & 1 \end{matrix}\right) \left(\begin{matrix} H_{2n+1}^{2p}(\mathbf{g}^1) & \mathbf{1}_{(2n+1)\times 1}\\ \mathbf{1}_{1\times (2n+1)} & 0\end{matrix}\right) \left(\begin{matrix}P_1 & \mathbf{0}_{(2n+1)\times 1}\\ \mathbf{0}_{1\times (2n+1)} & 1 \end{matrix}\right)\\
&\qquad=\left(\begin{matrix}P_1^t H_{2n+1}^{2p}(\mathbf{g}^1)P_1 & P_1^t\mathbf{1}_{(2n+1)\times 1}\\ \mathbf{1}_{1\times (2n+1)}P_1 & 0\end{matrix}\right).
\end{align*} Thus using \eqref{PHP221} we have 
\begin{align*} &\left(\begin{matrix}P_1^t & \mathbf{0}_{(2n+1)\times 1}\\ \mathbf{0}_{1\times (2n+1)} & 1 \end{matrix}\right) \overline{H_{2n+1}^{2p}(\mathbf{g}^1)} \left(\begin{matrix}P_1 & \mathbf{0}_{(2n+1)\times 1}\\ \mathbf{0}_{1\times (2n+1)} & 1 \end{matrix}\right) \\
&\qquad\qquad\qquad=\left(\begin{matrix}\mathbf{1}_{(n+1)\times (n+1)}+ H_{n+1}^{p}(\mathbf{g}^1) & \mathbf{1}_{(n+1)\times n} & \mathbf{1}_{(n+1)\times 1}\\ \mathbf{1}_{n\times (n+1)} & \mathbf{1}_{n\times n}+ H_{n}^{p+1}(\mathbf{g}^1) & \mathbf{1}_{n\times 1}\\ \mathbf{1}_{1\times (n+1)} & \mathbf{1}_{1\times n} & 0\end{matrix}\right).
\end{align*} Just as before, we have $|(0)|=0$ and $|\overline{(0)}|=-1$, and so again by one of the above lemmas \begin{align*}|\overline{H_{2n+1}^{2p}(\mathbf{g}^1)}|
&=|\mathbf{1}_{(n+1)\times (n+1)}+ H_{n+1}^{p}(\mathbf{g}^1)|\cdot |\overline{\mathbf{1}_{n\times n}+ H_{n}^{p+1}(\mathbf{g}^1)}|\\
&\qquad+|\overline{\mathbf{1}_{(n+1)\times (n+1)}+ H_{n+1}^{p}(\mathbf{g}^1)}|\cdot |\mathbf{1}_{n\times n}+ H_{n}^{p+1}(\mathbf{g}^1)|\\
&\qquad\qquad+2|\overline{\mathbf{1}_{(n+1)\times (n+1)}+ H_{n+1}^{p}(\mathbf{g}^1)}|\cdot |\overline{\mathbf{1}_{n\times n}+ H_{n}^{p+1}(\mathbf{g}^1)}|\\
&=\left(|H_{n+1}^{p}(\mathbf{g}^1)|-|\overline{H_{n+1}^{p}(\mathbf{g}^1)}|\right)\cdot |\overline{H_{n}^{p+1}(\mathbf{g}^1)}|\\
&\qquad+|\overline{H_{n+1}^{p}(\mathbf{g}^1)}|\cdot\left(|H_{n}^{p+1}(\mathbf{g}^1)|-|\overline{H_{n}^{p+1}(\mathbf{g}^1)}|\right)\\
&\qquad\qquad+2  |\overline{H_{n+1}^{p}(\mathbf{g}^1)}|\cdot |\overline{H_{n}^{p+1}(\mathbf{g}^1)}|\\
&\equiv |H_{n+1}^{p}(\mathbf{g}^1)|\cdot |\overline{H_{n}^{p+1}(\mathbf{g}^1)}|+|\overline{H_{n+1}^{p}(\mathbf{g}^1)}|\cdot|H_{n}^{p+1}(\mathbf{g}^1)|.
\end{align*}

For (v), we have \begin{align*} \left(\begin{matrix} \mathbf{0}_{n\times n} & I_n\\ I_n &\mathbf{0}_{n\times n}\end{matrix}\right)P_1^tH_{2n}^{2p+1}(\mathbf{g}^1)P_1&=\left(\begin{matrix} K_{n}^{2p+2}(\mathbf{g}^1) & K_{n}^{2p+3}(\mathbf{g}^1)\\ K_{n}^{2p+1}(\mathbf{g}^1) &K_{n}^{2p+2}(\mathbf{g}^1) \end{matrix}\right)\\
&=\left(\begin{matrix} \mathbf{1}_{n\times n}+ H_{n}^{p+1}(\mathbf{g}^1) & \mathbf{1}_{n\times n}\\ \mathbf{1}_{n\times n} & \mathbf{1}_{n\times n}+ H_{n}^{p+1}(\mathbf{g}^1)\end{matrix}\right),
\end{align*} and so \begin{align*}|H_{2n}^{2p+1}(\mathbf{g}^1)|&=|\overline{\mathbf{1}_{n\times n}+ H_{n}^{p+1}(\mathbf{g}^1)}|^2-|\mathbf{1}_{n\times n}+ H_{n}^{p+1}(\mathbf{g}^1)|^2\\
&=|\overline{H_{n}^{p+1}(\mathbf{g}^1)}|^2-\left(|H_{n}^{p+1}(\mathbf{g}^1)|-|\overline{H_{n}^{p+1}(\mathbf{g}^1)}|\right)^2\\
&=-|H_{n}^{p+1}(\mathbf{g}^1)|^2+2|H_{n}^{p+1}(\mathbf{g}^1)|\cdot|\overline{H_{n}^{p+1}(\mathbf{g}^1)}|\\
&\equiv |H_{n}^{p+1}(\mathbf{g}^1)|.
\end{align*}

For (vi), similar to (ii) we have using \eqref{APWW8} that \begin{align*} \left(\begin{matrix}P_1^t & \mathbf{0}_{2n\times 1}\\ \mathbf{0}_{1\times 2n} & 1 \end{matrix}\right) &\overline{H_{2n}^{2p+1}(\mathbf{g}^1)} \left(\begin{matrix}P_1 & \mathbf{0}_{2n\times 1}\\ \mathbf{0}_{1\times 2n} & 1 \end{matrix}\right) \\
&=  \left(\begin{matrix}P_1^t & \mathbf{0}_{2n\times 1}\\ \mathbf{0}_{1\times 2n} & 1 \end{matrix}\right) \left(\begin{matrix} H_{2n}^{2p+1}(\mathbf{g}^1) & \mathbf{1}_{2n\times 1}\\ \mathbf{1}_{1\times 2n} & 0\end{matrix}\right) \left(\begin{matrix}P_1 & \mathbf{0}_{2n\times 1}\\ \mathbf{0}_{1\times 2n} & 1 \end{matrix}\right)\\
&=\left(\begin{matrix}P_1^t H_{2n}^{2p+1}(\mathbf{g}^1)P_1 & P_1^t\mathbf{1}_{2n\times 1}\\ \mathbf{1}_{1\times 2n}P_1 & 0\end{matrix}\right)\\
&=\left(\begin{matrix}\mathbf{1}_{n\times n} &\mathbf{1}_{n\times n}+ H_{n}^{p+1}(\mathbf{g}^1) &  \mathbf{1}_{n\times 1}\\ \mathbf{1}_{n\times n}+ H_{n}^{p+1}(\mathbf{g}^1) & \mathbf{1}_{n\times n} & \mathbf{1}_{n\times 1}\\ \mathbf{1}_{1\times n} & \mathbf{1}_{1\times n} & 0\end{matrix}\right).
\end{align*} Thus we have that \begin{align*}\left(\begin{matrix} \mathbf{0}_{n\times n} & I_n & \mathbf{0}_{n\times 1}\\
I_n &\mathbf{0}_{n\times n} & \mathbf{0}_{n\times 1}\\
\mathbf{0}_{1\times n} & \mathbf{0}_{1\times n} & 1\end{matrix}\right)&\left(\begin{matrix}P_1^t & \mathbf{0}_{2n\times 1}\\ \mathbf{0}_{1\times 2n} & 1 \end{matrix}\right) \overline{H_{2n}^{2p+1}(\mathbf{g}^1)} \left(\begin{matrix}P_1 & \mathbf{0}_{2n\times 1}\\ \mathbf{0}_{1\times 2n} & 1 \end{matrix}\right) \\
&=\left(\begin{matrix}\mathbf{1}_{n\times n}+ H_{n}^{p+1}(\mathbf{g}^1) & \mathbf{1}_{n\times n} &  \mathbf{1}_{n\times 1}\\  \mathbf{1}_{n\times n} &\mathbf{1}_{n\times n}+ H_{n}^{p+1}(\mathbf{g}^1) & \mathbf{1}_{n\times 1}\\ \mathbf{1}_{1\times n} & \mathbf{1}_{1\times n} & 0\end{matrix}\right).
 \end{align*} Since $$\left|\begin{matrix} \mathbf{0}_{n\times n} & I_n & \mathbf{0}_{n\times 1}\\
I_n &\mathbf{0}_{n\times n} & \mathbf{0}_{n\times 1}\\
\mathbf{0}_{1\times n} & \mathbf{0}_{1\times n} & 1\end{matrix}\right|=-1,$$ this gives, applying Lemma (3 by 3), that $$|\overline{H_{2n}^{2p+1}(\mathbf{g}^1)}|=-2|\mathbf{1}_{n\times n}+ H_{n}^{p+1}(\mathbf{g}^1)|\cdot|\overline{\mathbf{1}_{n\times n}+ H_{n}^{p+1}(\mathbf{g}^1)}|-2|\overline{\mathbf{1}_{n\times n}+ H_{n}^{p+1}(\mathbf{g}^1)}|^2\equiv 0.$$

For (vii') and (vii'') we will use the well--known (see \cite[Remark 2.1]{APWW1998} or \cite[Page 96]{B1980}) recurrence \begin{equation}\label{alice}|H_{n}^{p}(\mathbf{u})|\cdot|H_{n}^{p+2}(\mathbf{u})|-|H_{n}^{p+1}(\mathbf{u})|^2=|H_{n-1}^{p+2}(\mathbf{u})|\cdot|H_{n+1}^{p}(\mathbf{u})|,\end{equation} with $p\mapsto 2p$, $n\mapsto 2n+1$ and $\mathbf{u}=\mathbf{g}^1$, to get $$|H_{2n+1}^{2p}(\mathbf{g}^1)|\cdot|H_{2n+1}^{2(p+1)}(\mathbf{g}^1)|-|H_{2n+1}^{2p+1}(\mathbf{g}^1)|^2=|H_{2n}^{2(p+1)}(\mathbf{g}^1)|\cdot|H_{2(n+1)}^{2p}(\mathbf{g}^1)|.$$ Solving for $|H_{2n+1}^{2p+1}(\mathbf{g}^1)|$ and remembering that we are always taking everything modulo $2$, this gives $$|H_{2n+1}^{2p+1}(\mathbf{g}^1)|\equiv |H_{2n+1}^{2p}(\mathbf{g}^1)|\cdot|H_{2n+1}^{2(p+1)}(\mathbf{g}^1)|-|H_{2n}^{2(p+1)}(\mathbf{g}^1)|\cdot|H_{2(n+1)}^{2p}(\mathbf{g}^1)|.$$ We now must differentiate between the cases $p=0$ and $p\geqslant 1$. 

When $p=0$, substituting in the identities of parts (iii'') and (ii''), we have \begin{align*}|H_{2n+1}^{1}(\mathbf{g}^1)|
&\equiv \Big[\Big(|H_{n+1}^{0}(\mathbf{g}^0)|\cdot|H_{n}^{1}(\mathbf{g}^1)|-|\overline{H_{n+1}^{0}(\mathbf{g}^0)}|\cdot|H_{n}^{1}(\mathbf{g}^1)|\\
&\qquad\qquad\qquad\qquad-|H_{n+1}^{0}(\mathbf{g}^0)|\cdot|\overline{H_{n}^{1}(\mathbf{g}^1)}|\Big)\\
&\times\Big(|H_{n+1}^{1}(\mathbf{g}^1)|\cdot|H_{n}^{2}(\mathbf{g}^1)|-|\overline{H_{n+1}^{1}(\mathbf{g}^1)}|\cdot|H_{n}^{2}(\mathbf{g}^1)|\\
&\qquad\qquad\qquad\qquad-|H_{n+1}^{1}(\mathbf{g}^1)|\cdot|\overline{H_{n}^{2}(\mathbf{g}^1)}|\Big)\Big]\\
&-\Big[\Big(|H_{n}^{1}(\mathbf{g}^1)|\cdot|H_{n}^{2}(\mathbf{g}^1)|-|\overline{H_{n}^{1}(\mathbf{g}^1)}|\cdot|H_{n}^{2}(\mathbf{g}^1)|-|H_{n}^{1}(\mathbf{g}^1)|\cdot|\overline{H_{n}^{2}(\mathbf{g}^1)}|\Big)\\
&\times\Big(|H_{n+1}^{0}(\mathbf{g}^0)|\cdot|H_{n+1}^{1}(\mathbf{g}^1)|-|\overline{H_{n+1}^{0}(\mathbf{g}^0)}|\cdot|H_{n+1}^{1}(\mathbf{g}^1)|\\
&\qquad\qquad\qquad\qquad-|H_{n+1}^{0}(\mathbf{g}^0)|\cdot|\overline{H_{n+1}^{1}(\mathbf{g}^1)}|\Big)\Big].
\end{align*} The similar result holds for $|H_{2n+1}^{1}(\mathbf{g}^0)|$ by replacing $\mathbf{g}^1$ with $\mathbf{g}^0$ in the above argument. This proves~(vii').

For $p\geqslant 1$, substituting in the identities of parts (iii'') and (ii'') and simplifying, this becomes \begin{align*}|H_{2n+1}^{2p+1}(\mathbf{g}^1)|
&\equiv \Big[\Big(|H_{n+1}^{p}(\mathbf{g}^1)|\cdot|H_{n}^{p+1}(\mathbf{g}^1)|-|\overline{H_{n+1}^{p}(\mathbf{g}^1)}|\cdot|H_{n}^{p+1}(\mathbf{g}^1)|\\
&\qquad\qquad\qquad\qquad-|H_{n+1}^{p}(\mathbf{g}^1)|\cdot|\overline{H_{n}^{p+1}(\mathbf{g}^1)}|\Big)\\
&\times\Big(|H_{n+1}^{p+1}(\mathbf{g}^1)|\cdot|H_{n}^{p+2}(\mathbf{g}^1)|-|\overline{H_{n+1}^{p+1}(\mathbf{g}^1)}|\cdot|H_{n}^{p+2}(\mathbf{g}^1)|\\
&\qquad\qquad\qquad\qquad-|H_{n+1}^{p+1}(\mathbf{g}^1)|\cdot|\overline{H_{n}^{p+2}(\mathbf{g}^1)}|\Big)\Big]\\
&-\Big[\Big(|H_{n}^{p+1}(\mathbf{g}^1)|\cdot|H_{n}^{p+2}(\mathbf{g}^1)|-|\overline{H_{n}^{p+1}(\mathbf{g}^1)}|\cdot|H_{n}^{p+2}(\mathbf{g}^1)|\\
&\qquad\qquad\qquad\qquad-|H_{n}^{p+1}(\mathbf{g}^1)|\cdot|\overline{H_{n}^{p+2}(\mathbf{g}^1)}|\Big)\\
&\times\Big(|H_{n+1}^{p}(\mathbf{g}^1)|\cdot|H_{n+1}^{p+1}(\mathbf{g}^1)|-|\overline{H_{n+1}^{p}(\mathbf{g}^1)}|\cdot|H_{n+1}^{p+1}(\mathbf{g}^1)|\\
&\qquad\qquad\qquad\qquad-|H_{n+1}^{p}(\mathbf{g}^1)|\cdot|\overline{H_{n+1}^{p+1}(\mathbf{g}^1)}|\Big)\Big]\\
&=\Big(|\overline{H_{n+1}^{p}(\mathbf{g}^1)}|\cdot|H_{n}^{p+2}(\mathbf{g}^1)|-|H_{n+1}^{p}(\mathbf{g}^1)|\cdot|\overline{H_{n}^{p+2}(\mathbf{g}^1)}|\Big)\\
&\qquad\qquad\times\Big(|H_{n}^{p+1}(\mathbf{g}^1)|\cdot|\overline{H_{n+1}^{p+1}(\mathbf{g}^1)}|-|\overline{H_{n}^{p+1}(\mathbf{g}^1)}|\cdot|H_{n+1}^{p+1}(\mathbf{g}^1)|\Big).
\end{align*} This proves (vii'').

For (viii') and (viii''), we note that for all $n,p$ and $\mathbf{u}$ we have, denoting $\mathbf{u}+1=\{u(n)+1\}_{n\geqslant 0}$, that $$H_n^p(\mathbf{u}+1)=\mathbf{1}_{n\times n}+H_n^p(\mathbf{u}).$$ Applying \eqref{alice} for the sequence $\mathbf{g}^1+1$ and doing everything modulo $2$, we have that \begin{multline*} |\mathbf{1}_{n\times n}+H_{n}^{p+1}(\mathbf{g}^1)|\equiv |\mathbf{1}_{n\times n}+H_{n}^{p}(\mathbf{g}^1)|\cdot|\mathbf{1}_{n\times n}+H_{n}^{p+2}(\mathbf{g}^1)|\\ +|\mathbf{1}_{(n-1)\times(n-1)}+H_{n-1}^{p+2}(\mathbf{g}^1)|\cdot|\mathbf{1}_{(n+1)\times(n+1)}+H_{n+1}^{p}(\mathbf{g}^1)|.\end{multline*} Now sending $p\mapsto 2p$ and $n\mapsto 2n+1$ yields \begin{multline*} |\mathbf{1}_{(2n+1)\times (2n+1)}+H_{2n+1}^{2p+1}(\mathbf{g}^1)|\\ \equiv |\mathbf{1}_{(2n+1)\times (2n+1)}+H_{2n+1}^{2p}(\mathbf{g}^1)|\cdot|\mathbf{1}_{(2n+1)\times (2n+1)}+H_{2n+1}^{2(p+1)}(\mathbf{g}^1)|\\ +|\mathbf{1}_{2n \times 2n}+H_{2n}^{2(p+1)}(\mathbf{g}^1)|\cdot|\mathbf{1}_{2(n+2)\times 2(n+2)}+H_{2(n+2)}^{2p}(\mathbf{g}^1)|.\end{multline*} Applying the identity from Lemma \ref{AbarA}(i) and solving for $|\overline{H_{2n+1}^{2p+1}(\mathbf{g}^1)}|,$ gives \begin{multline*} |\overline{H_{2n+1}^{2p+1}(\mathbf{g}^1)}| \equiv |H_{2n+1}^{2p+1}(\mathbf{g}^1)|\\ +\left(|H_{2n+1}^{2p}(\mathbf{g}^1)|+|\overline{H_{2n+1}^{2p}(\mathbf{g}^1)|}\right)\cdot\left(|H_{2n+1}^{2(p+1)}(\mathbf{g}^1)|+|\overline{H_{2n+1}^{2(p+1)}(\mathbf{g}^1)}|\right)\\ +\left(|H_{2n}^{2(p+1)}(\mathbf{g}^1)|+|\overline{H_{2n}^{2(p+1)}(\mathbf{g}^1)}|\right)\cdot\left(|H_{2(n+2)}^{2p}(\mathbf{g}^1)|+|\overline{H_{2(n+2)}^{2p}(\mathbf{g}^1)}|\right).\end{multline*} Now applying the results we have just proven from (i''), (ii''), (iii''), (iv''), and (vii'') we have that 
\begin{align}\label{barH2121} |\overline{H_{2n+1}^{2p+1}(\mathbf{g}^1)}| &\equiv \left(|H_{n}^{p+2}(\mathbf{g}^1)|\cdot|\overline{H_{n+1}^{p}(\mathbf{g}^1)}|-|H_{n+1}^{p}(\mathbf{g}^1)|\cdot|\overline{H_{n}^{p+2}(\mathbf{g}^1)}|\right)\\
\nonumber &\qquad\times\left(|H_{n}^{p+1}(\mathbf{g}^1)|\cdot|\overline{H_{n+1}^{p+1}(\mathbf{g}^1)}|-|H_{n+1}^{p+1}(\mathbf{g}^1)|\cdot|\overline{H_{n}^{p+1}(\mathbf{g}^1)}|\right)\\
\nonumber &\qquad+\left(|H_{2n+1}^{2p}(\mathbf{g}^1)|+|\overline{H_{2n+1}^{2p}(\mathbf{g}^1)}|\right)\cdot\left(|H_{n+1}^{p+1}(\mathbf{g}^1)|\cdot |H_{n}^{p+2}(\mathbf{g}^1)|\right)\\ 
\nonumber &\qquad+\left(|H_{n}^{p+1}(\mathbf{g}^1)|\cdot |H_{n}^{p+2}(\mathbf{g}^1)|\right)\cdot\left(|H_{2(n+2)}^{2p}(\mathbf{g}^1)|+|\overline{H_{2(n+2)}^{2p}(\mathbf{g}^1)}|\right).\end{align} We note differentiate between $p=0$ and $p\geqslant 1$.

If $p=0$, we apply (i'), (ii'), (iii'), and (iv') to equivalence \eqref{barH2121} to get 
\begin{align*} |\overline{H_{2n+1}^{1}(\mathbf{g}^1)}| &\equiv \left(|H_{n}^{2}(\mathbf{g}^1)|\cdot|\overline{H_{n+1}^{0}(\mathbf{g}^1)}|-|H_{n+1}^{0}(\mathbf{g}^1)|\cdot|\overline{H_{n}^{2}(\mathbf{g}^1)}|\right)\\
&\qquad\times\left(|H_{n}^{1}(\mathbf{g}^1)|\cdot|\overline{H_{n+1}^{1}(\mathbf{g}^1)}|-|H_{n+1}^{1}(\mathbf{g}^1)|\cdot|\overline{H_{n}^{1}(\mathbf{g}^1)}|\right)\\
&\qquad+|H_{n+1}^{0}(\mathbf{g}^0)|\cdot|H_{n}^{1}(\mathbf{g}^1)|\cdot|H_{n+1}^{1}(\mathbf{g}^1)|\cdot |H_{n}^{2}(\mathbf{g}^1)|\\ 
&\qquad+|H_{n}^{1}(\mathbf{g}^1)|\cdot |H_{n}^{2}(\mathbf{g}^1)|\cdot|H_{n+2}^{0}(\mathbf{g}^0)|\cdot|\overline{H_{n+2}^{1}(\mathbf{g}^1)}|,\end{align*} which proves (viii').

If $p\geqslant 1$, we apply (i''), (ii''), (iii''), and (iv'') to equivalence \eqref{barH2121} to get
\begin{align*} |\overline{H_{2n+1}^{2p+1}(\mathbf{g}^1)}| &\equiv \left(|H_{n}^{p+2}(\mathbf{g}^1)|\cdot|\overline{H_{n+1}^{p}(\mathbf{g}^1)}|-|H_{n+1}^{p}(\mathbf{g}^1)|\cdot|\overline{H_{n}^{p+2}(\mathbf{g}^1)}|\right)\\
&\qquad\times\left(|H_{n}^{p+1}(\mathbf{g}^1)|\cdot|\overline{H_{n+1}^{p+1}(\mathbf{g}^1)}|-|H_{n+1}^{p+1}(\mathbf{g}^1)|\cdot|\overline{H_{n}^{p+1}(\mathbf{g}^1)}|\right)\\
&\qquad+|H_{n+1}^{p}(\mathbf{g}^1)|\cdot|H_{n}^{p+1}(\mathbf{g}^1)|\cdot|H_{n+1}^{p+1}(\mathbf{g}^1)|\cdot |H_{n}^{p+2}(\mathbf{g}^1)|\\ 
&\qquad+|H_{n}^{p+1}(\mathbf{g}^1)|\cdot |H_{n}^{p+2}(\mathbf{g}^1)|\cdot|H_{n+2}^{p}(\mathbf{g}^1)|\cdot|H_{n+2}^{p+1}(\mathbf{g}^1)|,\end{align*} which proves (viii'') and completes the proof of the lemma. 
\end{proof}

We have the following corollary. Note that we have made enumerated the parts of the corollary to coincide with the enumeration of Lemma \ref{mainlemma}.

\begin{corollary}\label{012} For all $n\geqslant 1$, we have
\begin{quote} 
\begin{enumerate}
\item[(i')] $|H_{2n}^{0}(\mathbf{g}^1)|\equiv |H_{n}^{0}(\mathbf{g}^0)|\cdot|H_{n}^{1}(\mathbf{g}^1)|-|\overline{H_{n}^{0}(\mathbf{g}^0)}|\cdot|H_{n}^{1}(\mathbf{g}^1)|-|H_{n}^{0}(\mathbf{g}^0)|\cdot|\overline{H_{n}^{1}(\mathbf{g}^1)}|,$
\item[] $|H_{2n}^{0}(\mathbf{g}^0)|\equiv |H_{n}^{0}(\mathbf{g}^1)|\cdot|H_{n}^{1}(\mathbf{g}^1)|-|\overline{H_{n}^{0}(\mathbf{g}^1)}|\cdot|H_{n}^{1}(\mathbf{g}^1)|-|H_{n}^{0}(\mathbf{g}^1)|\cdot|\overline{H_{n}^{1}(\mathbf{g}^1)}|,$
\vspace{.1cm}
\item[(i'')] $|H_{2n}^{2}(\mathbf{g}^1)|\equiv |H_{n}^{1}(\mathbf{g}^1)|\cdot|H_{n}^{2}(\mathbf{g}^1)|-|\overline{H_{n}^{1}(\mathbf{g}^1)}|\cdot|H_{n}^{2}(\mathbf{g}^1)|-|H_{n}^{1}(\mathbf{g}^1)|\cdot|\overline{H_{n}^{2}(\mathbf{g}^1)}|,$
\vspace{.1cm}
\item[(ii')] $|\overline{H_{2n}^{0}(\mathbf{g}^1)}|\equiv |H_{n}^{0}(\mathbf{g}^0)|\cdot |\overline{H_{n}^{1}(\mathbf{g}^1)}|+|\overline{H_{n}^{0}(\mathbf{g}^0)}|\cdot|H_{n}^{1}(\mathbf{g}^1)|,$
\item[] $|\overline{H_{2n}^{0}(\mathbf{g}^0)}|\equiv |H_{n}^{0}(\mathbf{g}^1)|\cdot |\overline{H_{n}^{1}(\mathbf{g}^1)}|+|\overline{H_{n}^{0}(\mathbf{g}^1)}|\cdot|H_{n}^{1}(\mathbf{g}^1)|,$
\vspace{.2cm}
\item[(ii'')] $|\overline{H_{2n}^{2}(\mathbf{g}^1)}|\equiv |H_{n}^{1}(\mathbf{g}^1)|\cdot |\overline{H_{n}^{2}(\mathbf{g}^1)}|+|\overline{H_{n}^{1}(\mathbf{g}^1)}|\cdot|H_{n}^{2}(\mathbf{g}^1)|,$
\vspace{.2cm}
\item[(iii')] $|H_{2n+1}^{0}(\mathbf{g}^1)|\equiv |H_{n+1}^{0}(\mathbf{g}^0)|\cdot |H_{n}^{1}(\mathbf{g}^1)|-|\overline{H_{n+1}^{0}(\mathbf{g}^0)}|\cdot|H_{n}^{1}(\mathbf{g}^1)|$
\item[] $\qquad\qquad\qquad\qquad-|H_{n+1}^{0}(\mathbf{g}^0)|\cdot|\overline{H_{n}^{1}(\mathbf{g}^1)}|$
\item[] $|H_{2n+1}^{0}(\mathbf{g}^0)|=|H_{n+1}^{0}(\mathbf{g}^1)|\cdot |H_{n}^{1}(\mathbf{g}^1)|-|\overline{H_{n+1}^{0}(\mathbf{g}^1)}|\cdot|H_{n}^{1}(\mathbf{g}^1)|$
\item[] $\qquad\qquad\qquad\qquad-|H_{n+1}^{0}(\mathbf{g}^1)|\cdot|\overline{H_{n}^{1}(\mathbf{g}^1)}|$
\vspace{.2cm}
\item[(iii'')] $|H_{2n+1}^{2}(\mathbf{g}^1)|\equiv |H_{n+1}^{1}(\mathbf{g}^1)|\cdot |H_{n}^{2}(\mathbf{g}^1)|-|\overline{H_{n+1}^{1}(\mathbf{g}^1)}|\cdot|H_{n}^{2}(\mathbf{g}^1)|$
\item[] $\qquad\qquad\qquad\qquad-|H_{n+1}^{1}(\mathbf{g}^1)|\cdot|\overline{H_{n}^{2}(\mathbf{g}^1)}|,$
\vspace{.2cm}
\item[(iv')] $|\overline{H_{2n+1}^{0}(\mathbf{g}^1)}|\equiv |H_{n+1}^{0}(\mathbf{g}^0)|\cdot |\overline{H_{n}^{1}(\mathbf{g}^1)}|+|\overline{H_{n+1}^{0}(\mathbf{g}^0)}|\cdot|H_{n}^{1}(\mathbf{g}^1)|$
\vspace{.2cm}
\item[] $|\overline{H_{2n+1}^{0}(\mathbf{g}^0)}|\equiv |H_{n+1}^{0}(\mathbf{g}^1)|\cdot |\overline{H_{n}^{1}(\mathbf{g}^1)}|+|\overline{H_{n+1}^{0}(\mathbf{g}^1)}|\cdot|H_{n}^{1}(\mathbf{g}^1)|$,
\vspace{.2cm}
\item[(iv'')] $|\overline{H_{2n+1}^{2}(\mathbf{g}^1)}|\equiv |H_{n+1}^{1}(\mathbf{g}^1)|\cdot |\overline{H_{n}^{2}(\mathbf{g}^1)}|+|\overline{H_{n+1}^{1}(\mathbf{g}^1)}|\cdot|H_{n}^{2}(\mathbf{g}^1)|,$
\vspace{.2cm}
\item[(v)] $|H_{2n}^{1}(\mathbf{g}^1)|\equiv |H_{n}^{1}(\mathbf{g}^1)|,$
\vspace{.2cm}
\item[(vi)] $|\overline{H_{2n}^{1}(\mathbf{g}^1)}|\equiv 0,$

\newpage
\item[(vii')] $|H_{2n+1}^{1}(\mathbf{g}^1)|\equiv \Big[\Big(|H_{n+1}^{0}(\mathbf{g}^0)|\cdot|H_{n}^{1}(\mathbf{g}^1)|-|\overline{H_{n+1}^{0}(\mathbf{g}^0)}|\cdot|H_{n}^{1}(\mathbf{g}^1)|$
\item[] $\qquad\qquad\qquad\qquad-|H_{n+1}^{0}(\mathbf{g}^0)|\cdot|\overline{H_{n}^{1}(\mathbf{g}^1)}|\Big)$
\vspace{-.4cm}
\item[] \begin{align*}&\qquad\qquad\times\Big(|H_{n+1}^{1}(\mathbf{g}^1)|\cdot|H_{n}^{2}(\mathbf{g}^1)|-|\overline{H_{n+1}^{1}(\mathbf{g}^1)}|\cdot|H_{n}^{2}(\mathbf{g}^1)|\\
&\qquad\qquad\qquad-|H_{n+1}^{1}(\mathbf{g}^1)|\cdot|\overline{H_{n}^{2}(\mathbf{g}^1)}|\Big)\Big]\\
&\qquad\qquad-\Big[\Big(|H_{n}^{1}(\mathbf{g}^1)|\cdot|H_{n}^{2}(\mathbf{g}^1)|-|\overline{H_{n}^{1}(\mathbf{g}^1)}|\cdot|H_{n}^{2}(\mathbf{g}^1)|\\
&\qquad\qquad\qquad-|H_{n}^{1}(\mathbf{g}^1)|\cdot|\overline{H_{n}^{2}(\mathbf{g}^1)}|\Big)\\
&\qquad\qquad\times\Big(|H_{n+1}^{0}(\mathbf{g}^0)|\cdot|H_{n+1}^{1}(\mathbf{g}^1)|-|\overline{H_{n+1}^{0}(\mathbf{g}^0)}|\cdot|H_{n+1}^{1}(\mathbf{g}^1)|\\
&\qquad\qquad\qquad-|H_{n+1}^{0}(\mathbf{g}^0)|\cdot|\overline{H_{n+1}^{1}(\mathbf{g}^1)}|\Big)\Big],
\end{align*}
\begin{align*}|H_{2n+1}^{1}(\mathbf{g}^0)|&\equiv \Big[\Big(|H_{n+1}^{0}(\mathbf{g}^1)|\cdot|H_{n}^{1}(\mathbf{g}^1)|-|\overline{H_{n+1}^{0}(\mathbf{g}^1)}|\cdot|H_{n}^{1}(\mathbf{g}^1)|\\
&\qquad\qquad\qquad-|H_{n+1}^{0}(\mathbf{g}^1)|\cdot|\overline{H_{n}^{1}(\mathbf{g}^1)}|\Big)\\
&\qquad\qquad\times\Big(|H_{n+1}^{1}(\mathbf{g}^1)|\cdot|H_{n}^{2}(\mathbf{g}^1)|-|\overline{H_{n+1}^{1}(\mathbf{g}^1)}|\cdot|H_{n}^{2}(\mathbf{g}^1)|\\
&\qquad\qquad\qquad-|H_{n+1}^{1}(\mathbf{g}^1)|\cdot|\overline{H_{n}^{2}(\mathbf{g}^1)}|\Big)\Big]\\
&\qquad\qquad-\Big[\Big(|H_{n}^{1}(\mathbf{g}^1)|\cdot|H_{n}^{2}(\mathbf{g}^1)|-|\overline{H_{n}^{1}(\mathbf{g}^1)}|\cdot|H_{n}^{2}(\mathbf{g}^1)|\\
&\qquad\qquad\qquad-|H_{n}^{1}(\mathbf{g}^1)|\cdot|\overline{H_{n}^{2}(\mathbf{g}^1)}|\Big)\\
&\qquad\qquad\times\Big(|H_{n+1}^{0}(\mathbf{g}^1)|\cdot|H_{n+1}^{1}(\mathbf{g}^1)|-|\overline{H_{n+1}^{0}(\mathbf{g}^1)}|\cdot|H_{n+1}^{1}(\mathbf{g}^1)|\\
&\qquad\qquad\qquad-|H_{n+1}^{0}(\mathbf{g}^1)|\cdot|\overline{H_{n+1}^{1}(\mathbf{g}^1)}|\Big)\Big],
\end{align*}
\item[(viii')] $ |\overline{H_{2n+1}^{1}(\mathbf{g}^1)}| \equiv \left(|H_{n}^{2}(\mathbf{g}^1)|\cdot|\overline{H_{n+1}^{0}(\mathbf{g}^1)}|-|H_{n+1}^{0}(\mathbf{g}^1)|\cdot|\overline{H_{n}^{2}(\mathbf{g}^1)}|\right)$
\vspace{-.5cm} 
\item[] \begin{align*} &\qquad\qquad\qquad\times\left(|H_{n}^{1}(\mathbf{g}^1)|\cdot|\overline{H_{n+1}^{1}(\mathbf{g}^1)}|-|H_{n+1}^{1}(\mathbf{g}^1)|\cdot|\overline{H_{n}^{1}(\mathbf{g}^1)}|\right)\\
&\qquad\qquad\qquad+|H_{n+1}^{0}(\mathbf{g}^0)|\cdot|H_{n}^{1}(\mathbf{g}^1)|\cdot|H_{n+1}^{1}(\mathbf{g}^1)|\cdot |H_{n}^{2}(\mathbf{g}^1)|\\ 
&\qquad\qquad\qquad+|H_{n}^{1}(\mathbf{g}^1)|\cdot |H_{n}^{2}(\mathbf{g}^1)|\cdot|H_{n+2}^{0}(\mathbf{g}^0)|\cdot|\overline{H_{n+2}^{1}(\mathbf{g}^1)}|.\end{align*}
\end{enumerate}
\end{quote}
\end{corollary}

\begin{proof}[Proof of Theorem \ref{Hankelg}]  First let us note that for $p\geqslant 1$, we trivially have that $H_n^p(\mathbf{g}^0)=H_n^p(\mathbf{g}^1),$ so that we do not need to worry about proving separately the cases for $\mathbf{g}^1$ and $\mathbf{g}^0$ in this range. 

We easily check (say with MAPLE) that Theorem \ref{Hankelg} is true for $n\leqslant 12$. The rest of the proof now follows by breaking up the cases of $n$ modulo $6$; that is, check that the theorem is true for $n$ equal to $6k,6k+1,6k+2,6k+3,6k+4,$ and $6k+5$. We write here only the case when $n=6k$. All of the other cases follow {\em mutatis mutandis}.

To this end, suppose the theorem is true for all $6k<m$. If $m=6k$ for some $k$ then Corollary~\ref{012}(i') gives \begin{align*} |H_{12k}^{0}(\mathbf{g}^1)|&\equiv |H_{6k}^{0}(\mathbf{g}^0)|\cdot|H_{6k}^{1}(\mathbf{g}^1)|-|\overline{H_{6k}^{0}(\mathbf{g}^0)}|\cdot|H_{6k}^{1}(\mathbf{g}^1)|-|H_{6k}^{0}(\mathbf{g}^0)|\cdot|\overline{H_{6k}^{1}(\mathbf{g}^1)}|\\
&\equiv  1\cdot1 - 1\cdot1 - 1\cdot0 \equiv 0,
\end{align*} and \begin{align*} |H_{12k}^{0}(\mathbf{g}^0)|&=|H_{6k}^{0}(\mathbf{g}^1)|\cdot|H_{6k}^{1}(\mathbf{g}^1)|-|\overline{H_{6k}^{0}(\mathbf{g}^1)}|\cdot|H_{6k}^{1}(\mathbf{g}^1)|-|H_{6k}^{0}(\mathbf{g}^1)|\cdot|\overline{H_{6k}^{1}(\mathbf{g}^1)}|\\
&\equiv  0\cdot1 - 1\cdot1 - 0\cdot0\equiv 1.
\end{align*} 

Corollary \ref{012}(i'') gives \begin{align*} |H_{12k}^{2}(\mathbf{g}^1)|&\equiv |H_{6k}^{1}(\mathbf{g}^1)|\cdot|H_{6k}^{2}(\mathbf{g}^1)|-|\overline{H_{6k}^{1}(\mathbf{g}^1)}|\cdot|H_{6k}^{2}(\mathbf{g}^1)|-|H_{6k}^{1}(\mathbf{g}^1)|\cdot|\overline{H_{6k}^{2}(\mathbf{g}^1)}|\\
&\equiv  1\cdot1 - 0\cdot1 - 1\cdot0 \equiv 1.
\end{align*}

Corollary \ref{012}(ii') gives \begin{align*}|\overline{H_{12k}^{0}(\mathbf{g}^1)}|&\equiv |H_{6k}^{0}(\mathbf{g}^0)|\cdot |\overline{H_{6k}^{1}(\mathbf{g}^1)}|+|\overline{H_{6k}^{0}(\mathbf{g}^0)}|\cdot|H_{6k}^{1}(\mathbf{g}^1)|\\
&\equiv  1\cdot0 + 1\cdot1  \equiv 1,
\end{align*} and \begin{align*} |\overline{H_{12k}^{0}(\mathbf{g}^0)}|&\equiv |H_{6k}^{0}(\mathbf{g}^1)|\cdot |\overline{H_{6k}^{1}(\mathbf{g}^1)}|+|\overline{H_{6k}^{0}(\mathbf{g}^1)}|\cdot|H_{6k}^{1}(\mathbf{g}^1)|\\
&\equiv  0\cdot0 + 1\cdot1 \equiv 1.
\end{align*}

Corollary \ref{012}(ii'') gives \begin{align*} |\overline{H_{12k}^{2}(\mathbf{g}^1)}|&\equiv |H_{6k}^{1}(\mathbf{g}^1)|\cdot |\overline{H_{6k}^{2}(\mathbf{g}^1)}|+|\overline{H_{6k}^{1}(\mathbf{g}^1)}|\cdot|H_{6k}^{2}(\mathbf{g}^1)|\\
&\equiv 1 \cdot0 + 0\cdot1  \equiv 0.
\end{align*}

Corollary \ref{012}(iii') gives \begin{align*}|H_{12k+1}^{0}(\mathbf{g}^1)|&\equiv |H_{6k+1}^{0}(\mathbf{g}^0)|\cdot |H_{6k}^{1}(\mathbf{g}^1)|-|\overline{H_{6k+1}^{0}(\mathbf{g}^0)}|\cdot|H_{6k}^{1}(\mathbf{g}^1)|\\
&\qquad\qquad-|H_{6k+1}^{0}(\mathbf{g}^0)|\cdot|\overline{H_{6k}^{1}(\mathbf{g}^1)}|\\
&\equiv  0\cdot1 - 1\cdot1 - 0\cdot0  \equiv 1,
\end{align*} and \begin{align*}|H_{12k+1}^{0}(\mathbf{g}^0)|&\equiv |H_{6k+1}^{0}(\mathbf{g}^1)|\cdot |H_{6k}^{1}(\mathbf{g}^1)|-|\overline{H_{6k+1}^{0}(\mathbf{g}^1)}|\cdot|H_{6k}^{1}(\mathbf{g}^1)|\\
&\qquad\qquad-|H_{6k+1}^{0}(\mathbf{g}^1)|\cdot|\overline{H_{6k}^{1}(\mathbf{g}^1)}|\\
&\equiv  1\cdot1 - 1\cdot1 - 1\cdot0  \equiv 0.
\end{align*}

Corollary \ref{012}(iii'') gives \begin{align*} |H_{12k+1}^{2}(\mathbf{g}^1)|&\equiv |H_{6k+1}^{1}(\mathbf{g}^1)|\cdot |H_{6k}^{2}(\mathbf{g}^1)|-|\overline{H_{6k+1}^{1}(\mathbf{g}^1)}|\cdot|H_{6k}^{2}(\mathbf{g}^1)|\\
&\qquad\qquad-|H_{6k+1}^{1}(\mathbf{g}^1)|\cdot|\overline{H_{6k}^{2}(\mathbf{g}^1)}|\\
&\equiv 1\cdot1 - 1\cdot1 - 1\cdot0 \equiv 0.
\end{align*}

Corollary \ref{012}(iv') gives \begin{align*} |\overline{H_{12k+1}^{0}(\mathbf{g}^1)}|&\equiv |H_{6k+1}^{0}(\mathbf{g}^0)|\cdot |\overline{H_{6k}^{1}(\mathbf{g}^1)}|+|\overline{H_{6k+1}^{0}(\mathbf{g}^0)}|\cdot|H_{6k}^{1}(\mathbf{g}^1)|\\
&\equiv  0\cdot0 + 1\cdot1 \equiv 1.
\end{align*} and \begin{align*} |\overline{H_{12k+1}^{0}(\mathbf{g}^0)}|&\equiv |H_{6k+1}^{0}(\mathbf{g}^1)|\cdot |\overline{H_{6k}^{1}(\mathbf{g}^1)}|+|\overline{H_{6k+1}^{0}(\mathbf{g}^1)}|\cdot|H_{6k}^{1}(\mathbf{g}^1)|\\
&\equiv 1\cdot0 + 1\cdot1 \equiv 1.
\end{align*}

Corollary \ref{012}(iv'') gives \begin{align*} |\overline{H_{12k+1}^{2}(\mathbf{g}^1)}|&\equiv |H_{6k+1}^{1}(\mathbf{g}^1)|\cdot |\overline{H_{6k}^{2}(\mathbf{g}^1)}|+|\overline{H_{6k+1}^{1}(\mathbf{g}^1)}|\cdot|H_{6k}^{2}(\mathbf{g}^1)|\\
&\equiv 1\cdot0 + 1\cdot1 \equiv 1.
\end{align*}

Corollary \ref{012}(v) gives $|H_{12k}^{1}(\mathbf{g}^1)|\equiv |H_{6k}^{1}(\mathbf{g}^1)|\equiv 1.$

\vspace{.2cm}
Corollary \ref{012}(vi) gives $|\overline{H_{12k}^{1}(\mathbf{g}^1)}|\equiv 0.$

Corollary \ref{012}(vii') gives \begin{align*} |H_{12k+1}^{1}&(\mathbf{g}^1)|\equiv \Big[\Big(|H_{6k+1}^{0}(\mathbf{g}^0)|\cdot|H_{6k}^{1}(\mathbf{g}^1)|-|\overline{H_{6k+1}^{0}(\mathbf{g}^0)}|\cdot|H_{6k}^{1}(\mathbf{g}^1)|\\
&\qquad\qquad\qquad-|H_{6k+1}^{0}(\mathbf{g}^0)|\cdot|\overline{H_{6k}^{1}(\mathbf{g}^1)}|\Big)\\
&\qquad\times\Big(|H_{6k+1}^{1}(\mathbf{g}^1)|\cdot|H_{6k}^{2}(\mathbf{g}^1)|-|\overline{H_{6k+1}^{1}(\mathbf{g}^1)}|\cdot|H_{6k}^{2}(\mathbf{g}^1)|\\
&\qquad\qquad\qquad-|H_{6k+1}^{1}(\mathbf{g}^1)|\cdot|\overline{H_{6k}^{2}(\mathbf{g}^1)}|\Big)\Big]\\
&\qquad-\Big[\Big(|H_{6k}^{1}(\mathbf{g}^1)|\cdot|H_{6k}^{2}(\mathbf{g}^1)|-|\overline{H_{6k}^{1}(\mathbf{g}^1)}|\cdot|H_{6k}^{2}(\mathbf{g}^1)|-|H_{6k}^{1}(\mathbf{g}^1)|\cdot|\overline{H_{6k}^{2}(\mathbf{g}^1)}|\Big)\\
&\qquad\times\Big(|H_{6k+1}^{0}(\mathbf{g}^0)|\cdot|H_{6k+1}^{1}(\mathbf{g}^1)|-|\overline{H_{6k+1}^{0}(\mathbf{g}^0)}|\cdot|H_{6k+1}^{1}(\mathbf{g}^1)|\\
&\qquad\qquad\qquad-|H_{6k+1}^{0}(\mathbf{g}^0)|\cdot|\overline{H_{6k}^{1}(\mathbf{g}^1)}|\Big)\Big]\\
&\qquad\equiv ( 0\cdot1 - 1\cdot1 - 0\cdot0 )( 1\cdot1 - 1\cdot1 - 1\cdot0 )\\
&\qquad\qquad\qquad-( 1\cdot1 - 0\cdot1 - 1\cdot0 )( 0\cdot1 - 1\cdot1 - 0\cdot1 )\\
&\qquad\equiv 1,
\end{align*} and 
\begin{align*} |H_{12k+1}^{1}&(\mathbf{g}^0)|\equiv \Big[\Big(|H_{6k+1}^{0}(\mathbf{g}^1)|\cdot|H_{6k}^{1}(\mathbf{g}^1)|-|\overline{H_{6k+1}^{0}(\mathbf{g}^1)}|\cdot|H_{6k}^{1}(\mathbf{g}^1)|\\
&\qquad\qquad\qquad-|H_{6k+1}^{0}(\mathbf{g}^1)|\cdot|\overline{H_{6k}^{1}(\mathbf{g}^1)}|\Big)\\
&\qquad\times\Big(|H_{6k+1}^{1}(\mathbf{g}^1)|\cdot|H_{6k}^{2}(\mathbf{g}^1)|-|\overline{H_{6k+1}^{1}(\mathbf{g}^1)}|\cdot|H_{6k}^{2}(\mathbf{g}^1)|\\
&\qquad\qquad\qquad-|H_{6k+1}^{1}(\mathbf{g}^1)|\cdot|\overline{H_{6k}^{2}(\mathbf{g}^1)}|\Big)\Big]\\
&\qquad-\Big[\Big(|H_{6k}^{1}(\mathbf{g}^1)|\cdot|H_{6k}^{2}(\mathbf{g}^1)|-|\overline{H_{6k}^{1}(\mathbf{g}^1)}|\cdot|H_{6k}^{2}(\mathbf{g}^1)|\\
&\qquad\qquad\qquad-|H_{6k}^{1}(\mathbf{g}^1)|\cdot|\overline{H_{6k}^{2}(\mathbf{g}^1)}|\Big)\\
&\qquad\times\Big(|H_{6k+1}^{0}(\mathbf{g}^1)|\cdot|H_{6k+1}^{1}(\mathbf{g}^1)|-|\overline{H_{6k+1}^{0}(\mathbf{g}^1)}|\cdot|H_{6k+1}^{1}(\mathbf{g}^1)|\\
&\qquad\qquad\qquad-|H_{6k+1}^{0}(\mathbf{g}^1)|\cdot|\overline{H_{6k}^{1}(\mathbf{g}^1)}|\Big)\Big]\\
&\qquad\equiv ( 1\cdot1 - 1\cdot1 - 1\cdot0 )( 1\cdot1 - 1\cdot1 - 1\cdot0 )\\
&\qquad\qquad\qquad-( 1\cdot1 - 0\cdot1 - 1\cdot0 )( 1\cdot1 - 1\cdot1 - 1\cdot1 )\\
&\qquad\equiv 1.
\end{align*}

Corollary \ref{012}(viii') gives 
\begin{align*} |\overline{H_{12k+1}^{1}(\mathbf{g}^1)}| &\equiv \left(|H_{6k}^{2}(\mathbf{g}^1)|\cdot|\overline{H_{6k+1}^{0}(\mathbf{g}^1)}|-|H_{6k+1}^{0}(\mathbf{g}^1)|\cdot|\overline{H_{6k}^{2}(\mathbf{g}^1)}|\right)\\ 
&\qquad\qquad\times\left(|H_{6k}^{1}(\mathbf{g}^1)|\cdot|\overline{H_{6k+1}^{1}(\mathbf{g}^1)}|-|H_{6k+1}^{1}(\mathbf{g}^1)|\cdot|\overline{H_{6k}^{1}(\mathbf{g}^1)}|\right)\\
&\qquad\qquad+|H_{6k+1}^{0}(\mathbf{g}^0)|\cdot|H_{6k}^{1}(\mathbf{g}^1)|\cdot|H_{6k+1}^{1}(\mathbf{g}^1)|\cdot |H_{6k}^{2}(\mathbf{g}^1)|\\ 
&\qquad\qquad+|H_{6k}^{1}(\mathbf{g}^1)|\cdot |H_{6k}^{2}(\mathbf{g}^1)|\cdot|H_{6k+2}^{0}(\mathbf{g}^0)|\cdot|\overline{H_{6k+2}^{1}(\mathbf{g}^1)}|\\
&\equiv ( 1\cdot1 - 1\cdot0)( 1\cdot1 - 1\cdot0)+ 0\cdot1 \cdot1 \cdot0 + 1\cdot0 \cdot1 \cdot0\\
&\equiv 1.\end{align*}
\end{proof} 

We can easily relate Theorem \ref{Hankelg} to give a similar result for the sequence $\mathbf{f}$ of coefficients of the series $\C{F}(z)=\sum_{n\geqslant 0} z^{2^n}(1+z^{2^n})^{-1}.$ 

\begin{corollary}\label{Hankelh} Let $\mathbf{h}=\{h(n)\}_{n\geqslant 1}$ be a sequence which is equivalent modulo $2$ to $\mathbf{g}$; that is, $h(n)\equiv g(n)\ ({\rm mod}\ 2).$ Then $|H_n^1(\mathbf{h})|$ is nonzero for all $n\geqslant 1$. In particular, the determinant of $|H_n^1(\mathbf{f})|$ is nonzero for all $n\geqslant 1$.
\end{corollary}

\begin{proof} It is enough to note that since $h(n)\equiv g(n)\ ({\rm mod}\ 2)$ for all $n\geqslant 1$, and so modulo $2$ we have $H_n^1(\mathbf{h})\equiv H_n^1(\mathbf{g}^1)$.
\end{proof}

\begin{proof}[Proof of Theorem \ref{mainH}] This is a direct consequence of Theorem \ref{Hankelg} and Corollary~\ref{Hankelh}.
\end{proof}

\section{Rational approximation of values of Golomb's series}\label{Sec3}

Given an analytic function $F(z)$, the rational function $R(x)$, with the degree of the numerator bounded by $m$ and the degree of the denominator bounded by $n$, is the $[m/n]_F$ {\em Pad\'e approximant} to $F(z)$ provided $$F(z)-R(z)=O(z^{m+n+1}).$$

We will need the following lemma connecting Hankel determinants to Pad\'e approximants (see \cite[Page 35]{B1980}).

\begin{lemma}[Brezinski \cite{B1980}]\label{HP} Let $\mathbf{c}=\{c(n)\}_{n\geqslant 0}$ and $\C{C}(z)=\sum_{n\geqslant 0}c(n)z^n\in\B{Z}[[z]].$ If $\det H_k^0(\mathbf{c})\neq 0$ for all $k\geqslant 1$, then the Pad\'e approximant $[k-1/k]_\C{C}$ exists and satisfies $$\C{C}(z)-[k-1/k]_\C{C}=\frac{\det H_{k+1}^0(\mathbf{c})}{\det H_k^0(\mathbf{c})}z^{2k}+O(z^{2k+1}).$$
\end{lemma}

An immediate consequence of Theorem \ref{Hankelg} and the Lemma \ref{HP} is the following lemma.

\begin{lemma}\label{Hg} Let $k\geqslant 1$, $\mathbf{h}=\{h(n)\}_{n\geqslant 1}$ be a sequence for which $h(n)\equiv g(n)\ ({\rm mod}\ 2)$, and $\C{H}(z):=\sum_{n\geqslant 1}h(n)z^n$. Then there exists a nonzero rational number $h_k$ and polynomials $P_k(z),Q_k(z)\in\B{Z}[z]$ with degrees bounded above by $k$ such that $$\C{H}(z)-\frac{P_k(z)}{Q_k(z)}=h_k z^{2k+1}+O(z^{2k+2}).$$
\end{lemma}

\begin{proof} Define the function $\C{H}'(z):=\C{H}(z)/z=\sum_{n\geqslant 0}h'(n)z^n$. Then $h'(n)=h(n+1)$ for all $n\geqslant 0$, and by Corollary \ref{Hankelh} for all $k\geqslant 1$ we have that $$H_{k}^1(\mathbf{h})=H_{k}^0(\mathbf{h'})\neq 0.$$ By Lemma \ref{HP} we have that the $[k-1/k]_{\C{H}'}$ exists and satisfies $$\C{H}'(z)-[k-1/k]_{\C{H}'}=\frac{\det H_{k+1}^1(\mathbf{h})}{\det H_k^1(\mathbf{h})}z^{2k}+O(z^{2k+1});$$ that is there exists polynomials $R_{k}(z),Q_{k}(z)\in\B{Z}[z]$, with $$\deg R_{k}(z)\leqslant k-1\qquad\mbox{and}\qquad\deg Q_{k}(z)\leqslant k,$$ such that \begin{equation}\label{HRQ} \C{H}'(z)-\frac{R_k(z)}{Q_k(z)}=h_k z^{2k}+O(z^{2k+1}),\end{equation} where we have set $h_k:={\det H_{k+1}^1(\mathbf{h})}/{\det H_k^1(\mathbf{h})}$ which is nonzero, since each of the numerator and denominator are nonzero. Multiplying both sides of \eqref{HRQ} by $z$ and denoting $P_k(z):=zR_k(z)$ proves the lemma. 
\end{proof}

Along with Lemma \ref{Hg} we will need the following result of Adamczewski and Rivoal \cite[Lemma 4.1]{AR2009} and a modification of a lemma of Bugeaud \cite[Lemma 2]{Bpreprint}.

\begin{lemma}[Adamczewski and Rivoal \cite{AR2009}]\label{ARLem} Let $\xi, \gd, \rho$ and $\gq$ be real numbers such that $0<\gd\leqslant\rho$ and $\gq\geqslant 1$. Let us assume that there exists a sequence $\{p_n/q_n\}_{n\geqslant 1}$ of rational numbers and some positive constants $c_0,c_1$ and $c_2$ such that both $$q_n<q_{n+1}\leqslant c_0 q_n^\gq,$$ and $$\frac{c_1}{q_n^{1+\rho}}\leqslant\left|\xi-\frac{p_n}{q_n}\right|\leqslant\frac{c_2}{q_n^{1+\gd}}.$$ Then we have that $$\mu(\xi)\leqslant (1+\rho)\frac{\gq}{\gd}.$$
\end{lemma}

\begin{lemma}[Modified Bugeaud]\label{Blem2} Let $K\geqslant 1$ and $n_0$ be positive integers. Let $(a_j)_{j\geqslant 1}$ be the increasing sequence of integers composed of all the numbers of the form $k2^n$, where $n\geqslant n_0$ and $k$ ranges over all the odd integers in $[2^{K-1}+1,2^K+1]$. Then $$a_{j+1}\leqslant \left(\frac{2^{K-1}+3}{2^{K-1}+1}\right)a_j.$$
\end{lemma}

\begin{proof} Let $n$ be large enough and consider the increasing sequence $(a_j)_{j\geqslant 1}$ of all integers of the form $k2^n$ where $k$ is an odd number in $[2^{K-1}+1,2^K+1].$ Note that for a given $j$ we have that for some $m$ and some odd number $a$ with $1\leqslant a\leqslant 2^{K-1}+1$ we have $a_j=2^m(2^{K-1}+a).$ We consider two cases.

If $a<2^{K-1}+1$, then $a_{j+1}\leqslant 2^m(2^{K-1}+a+2),$ so that $$\frac{a_{j+1}}{a_j}\leq \frac{2^{K-1}+a+2}{2^{K-1}+a}\leqslant \frac{2^{K-1}+3}{2^{K-1}+1}.$$ 

If $a=2^{K-1}+1$, then $a_{j+1}\leqslant 2^{m+1}(2^{K-1}+1)=2^{m}(2^{K}+2),$ so that $$\frac{a_{j+1}}{a_j}\leqslant \frac{2^{K}+2}{2^{K}+1}\leqslant \frac{2^{K-1}+3}{2^{K-1}+1}.$$ This proves the lemma.
\end{proof}

\begin{proof}[Proof of Theorem \ref{main}] Let $\epsilon\in\{-1,1\}$ and set $$\C{H}(z):=\sum_{n\geqslant 0}\frac{z^{2^n}}{1+\epsilon z^{2^{k}}}=\sum_{n\geqslant 1}h(n)z^n.$$ Note here that $\C{H}(z)$ satisfies the functional equation $$\C{H}(z^{2^m})=\C{H}(z)-\sum_{k=0}^{m-1}\frac{z^{2^{k}}}{1+\epsilon z^{2^{k}}}.$$
Applying Lemma \ref{Hg}, there exist polynomials $P_{k,0}(z),Q_{k,0}(z)\in\B{Z}[z]$, with both $\deg P_{k,0}(z)$ and $\deg Q_{k,0}(z)$ at most $k$, and a nonzero $h_k\in\B{Q}$ such that $$\C{H}(z)-\frac{P_{k,0}(z)}{Q_{k,0}(z)}=h_kz^{2k+1}+O(z^{2k+2}).$$ Thus sending $z\mapsto z^{2^m}$ we have that $$\C{H}(z^{2^m})-\frac{P_{k,0}(z^{2^m})}{Q_{k,0}(z^{2^m})}=h_kz^{2^{m}(2k+1)}+O(z^{2^m(2k+2)}),$$ and so using the functional equation for $\C{H}(z)$ we then have that $$\C{H}(z)-\left(\sum_{k=0}^{m-1}\frac{z^{2^{k}}}{1+\epsilon z^{2^{k}}}+\frac{P_{k,0}(z^{2^m})}{Q_{k,0}(z^{2^m})}\right)=h_kz^{2^{m}(2k+1)}+O(z^{2^m(2k+2)}).$$ Now define $P_{k,m}(z)$ and $Q_{k,m}(z)$ by $$\frac{P_{k,m}(z)}{Q_{k,m}(z)}:=\sum_{k=0}^{m-1}\frac{z^{2^{k}}}{1+\epsilon z^{2^{k}}}+\frac{P_{k,0}(z^{2^m})}{Q_{k,0}(z^{2^m})},$$ so that $$\C{H}(z)-\frac{P_{k,m}(z)}{Q_{k,m}(z)}=h_kz^{2^{m}(2k+1)}+O(z^{2^m(2k+2)}).$$ 

Now let $b\geqslant 2$ be an integer and as before set $z=\frac{1}{b}.$ Then for $\eps>0$, we have for large enough $m$, say $m\geqslant m_0(k)$, that $$(1-\eps)h_kb^{-2^{m}(2k+1)}\leqslant \left|\C{H}(1/b)-\frac{P_{k,m}(1/b)}{Q_{k,m}(1/b)}\right|\leqslant (1+\eps)h_kb^{-2^{m}(2k+1)}.$$ 

To get the degrees of $P_{k,m}(z)$ and $Q_{k,m}(z)$ we write \begin{equation}\label{PQsum}\frac{P(z)}{Q(z)}=\sum_{k=0}^{m-1}\frac{z^{2^{k}}}{1+\epsilon z^{2^{k}}}.\end{equation} Note that using \eqref{PQsum} it is immediate that $$\deg P(z)\leqslant \deg Q(z)\leqslant 2^{m}.$$ Using the definitions of $P(z)$ and $Q(z)$, we have that both $$\deg Q_{k,m}(z)=\deg Q(z)Q_{k,0}(z^{2^m})=\deg Q(z)+\deg Q_{k,0}(z^{2^m})\leqslant 2^{m}(k+1),$$ and $$\deg P_{k,m}(z)=\max\{\deg P(z)Q_{k,0}(z^{2^m}),P_{k,0}(z^{2^m})\}\leqslant 2^{m}(k+1).$$

Define the integers $$p_{k,m}:=b^{2^{m}(k+1)}P_{k,m}(1/b)$$ and $$q_{k,m}:=b^{2^{m}(k+1)}Q_{k,m}(1/b).$$

Since $h_k$ is nonzero there exist positive real constants $c_i(k)$ ($i=3,\ldots,6$) depending only on $k$ so that \begin{equation}\label{fqkm}c_3(k)b^{2^{m}(k+1)}\leqslant q_{k,m}\leqslant c_4(k)b^{2^{m}(k+1)},\end{equation} and \begin{equation}\label{Fpq}\frac{c_5(k)}{b^{2^{m}(2k+1)}}\leqslant\left|\C{H}(1/b)-\frac{p_{k,m}}{q_{k,m}}\right|\leqslant\frac{c_{6}(k)}{b^{2^{m}(2k+1)}}.\end{equation}

Thus by \eqref{fqkm} there are positive constants $c_{7}(k)$ and $c_{8}(k)$ such that $$\frac{c_{7}(k)}{q_{k,m}^{2}}\leqslant \frac{1}{b^{2^{m+1}(2k+1)}}\leqslant \frac{c_{8}(k)}{q_{k,m}^{2}}.$$ Applying this to \eqref{Fpq} yields $$\frac{c_{9}(k)}{q_{k,m}^{1+\frac{k}{k+1}}}\leqslant \left|\C{H}(1/b)-\frac{p_{k,m}}{q_{k,m}}\right|\leqslant\frac{c_{10}(k)}{q_{k,m}^{1+\frac{k}{k+1}}},$$ for some positive constants $c_{9}(k)$ and $c_{10}(k)$, from which we deduce that \begin{equation}\label{q2Fq2}\frac{c_{9}(k)}{q_{k,m}^{2}}\leqslant \left|\C{H}(1/b)-\frac{p_{k,m}}{q_{k,m}}\right|\leqslant\frac{c_{10}(k)}{q_{k,m}^{1+\frac{k}{k+1}}}.\end{equation}

Let $K\geqslant 1$ be an integer and denote by $m_0(k)$ the integer such that for $m\geqslant m_0(k)$ the sequence $\{q_{k,m}\}_{m\geqslant m_0(k)}$ is increasing. We define the sequence of positive integers $\{Q_{K,n}\}_{n\geqslant 1}$ as the sequence of all the integers $q_{k,m}$ with $k+1$ odd, $2^{K-1}+1\leqslant k+1\leqslant 2^K+1$, $m\geqslant m_0(k)$, put in increasing order. Then by Lemma \ref{Blem2} and \eqref{fqkm} there is an $n_0(K)$ and a positive constant $C_0(K)$ such that \begin{equation}\label{fQQQ}Q_{K,n}<Q_{K,n+1}\leqslant C_0(K)Q_{K,n}^{\frac{2^{K-1}+3}{2^{K-1}+1}}\end{equation} for all $n\geqslant n_0(K).$ By \eqref{q2Fq2}, there are positive integers $P_{K,n}$ and $Q_{K,n}$ and positive constants $C_1(K)$ and $C_2(K)$ such that \begin{equation}\label{FF}\frac{C_1(K)}{Q_{K,n}^{2}}\leqslant \left|\C{H}(1/b)-\frac{P_{K,n}}{Q_{K,n}}\right|\leqslant\frac{C_2(K)}{Q_{K,n}^{1+\frac{2^K}{2^K+1}}};\end{equation} here we have taken $P_{K,n}$ to be the $p_{k,m}$ associated to $q_{k,m}=Q_{K,n}$.

Applying Lemma \ref{ARLem}, using \eqref{FF} and \eqref{fQQQ}, we have that $$\mu\left(\C{H}(1/b)\right)\leqslant 2\left(\frac{2^{K-1}+3}{2^{K-1}+1}\right)\left(\frac{2^K+1}{2^K}\right).$$ Since $K$ can be taken arbitrarily large, we have that $\mu\left(\C{H}(1/b)\right)\leqslant 2$, and since $\C{H}(1/b)$ is transcendental (or even just using irrationality) we have that $$\mu\left(\C{H}(1/b)\right)=2.$$ 

Choosing $\epsilon=-1$ gives that $\mu\left(\C{G}(1/b)\right)=2$, and that choosing $\epsilon=1$ gives that $\mu\left(\C{F}(1/b)\right)=2$. This proves the theorem.
\end{proof}

\noindent\tbf{Acknowledgements} We wish to thank Yann Bugeaud, Kevin Hare, Cameron Stewart, and Jeffrey Shallit for helpful comments and conversations.

\bibliographystyle{amsalpha}

\begin{thebibliography}{APWW}

\bibitem[APWW]{APWW1998}
J.-P. Allouche, J.~Peyri{\`e}re, Z.-X. Wen, and Z.-Y. Wen, \emph{Hankel
  determinants of the {T}hue-{M}orse sequence}, Ann. Inst. Fourier (Grenoble)
  \textbf{48} (1998), no.~1, 1--27.

\bibitem[AR]{AR2009}
B.~Adamczewski and T.~Rivoal, \emph{Irrationality measures for some automatic
  real numbers}, Math. Proc. Cambridge Philos. Soc. \textbf{147} (2009), no.~3,
  659--678.

\bibitem[Bre]{B1980}
C.~Brezinski, \emph{Pad\'e-type approximation and general orthogonal
  polynomials}, International Series of Numerical Mathematics, vol.~50,
  Birkh\"auser Verlag, Basel, 1980.

\bibitem[Bug]{Bpreprint}
Y.~Bugeaud, \emph{On the rational approximation of the
  {T}hue--{M}orse--{M}ahler number}, preprint.

\bibitem[Coo]{C2011}
M.~Coons, \emph{Extension of some theorems of {W}.~{S}chwarz}, Canad. Math.
  Bull. (2011), doi:10.4153/CMB--2011--037--9.

\bibitem[Duv]{D2001}
D.~Duverney, \emph{Transcendence of a fast converging series of rational
  numbers}, Math. Proc. Cambridge Philos. Soc. \textbf{130} (2001), no.~2,
  193--207.

\bibitem[Gol]{G1963}
S.~W. Golomb, \emph{On the sum of the reciprocals of the {F}ermat numbers and
  related irrationalities}, Canad. J. Math. \textbf{15} (1963), 475--478.

\bibitem[Lio]{L1844}
J.~Liouville, \emph{Sur des classes tr\`es \'etendues de quantit\'es dont la
  valeur n'est ni alg\'ebrique, ni m\^eme r\'eductible \`a des irrationelles
  alg\'ebriques}, C.R. Acad. Sci. Paris \textbf{18} (1844), 883--885, 910--911.

\bibitem[Mah1]{M1929}
K.~Mahler, \emph{Arithmetische {E}igenschaften der {L}\"osungen einer {K}lasse
  von {F}unktionalgleichungen}, Math. Ann. \textbf{101} (1929), no.~1,
  342--366.

\bibitem[Mah2]{M1930a}
K.~Mahler, \emph{Arithmetische {E}igenschaften einer {K}lasse
  transzendental-transzendenter {F}unktionen}, Math. Z. \textbf{32} (1930),
  no.~1, 545--585.

\bibitem[Mah3]{M1930b}
K.~Mahler, \emph{Uber das {V}erschwinden von {P}otenzreihen mehrerer
  {V}er\"anderlichen in speziellen {P}unktfolgen}, Math. Ann. \textbf{103}
  (1930), no.~1, 573--587.

\bibitem[Rot]{R1955}
K.~F. Roth, \emph{Rational approximations to algebraic numbers}, Mathematika
  \textbf{2} (1955), 1--20; corrigendum, 168.

\bibitem[Sch]{S1967}
W.~Schwarz, \emph{Remarks on the irrationality and transcendence of certain
  series}, Math. Scand \textbf{20} (1967), 269--274.

\end{thebibliography}

\end{document}